\documentclass{acmsiggraph}

\usepackage[scaled=.92]{helvet}
\usepackage{times}
\usepackage{graphicx}
\usepackage{parskip}
\usepackage[labelfont=bf,textfont=it]{caption}
\onlineid{0}

\usepackage{algorithmic}
\usepackage{subfigure}
\usepackage{amsmath}
\usepackage{amssymb}

\newtheorem{lem}{Lemma} 
\newtheorem{cor}{Corollary} 
\newtheorem{thm}{Theorem}
    \newcommand{\ProofBox}[0]{\hfill \rule{2mm}{2mm}}
    \newenvironment{proof}{\begin{trivlist}\item[] {\sc Proof.}}{%
      \ifhmode\nolinebreak[4]~\ProofBox\else\ProofBox\fi \end{trivlist}}

\def\vec#1{\mbox{\boldmath $#1$}}

\title{Consistent Computation of First- and Second-Order\\
Differential Quantities for Surface Meshes}
\author{
Xiangmin Jiao\thanks{e-mail: jiao@ams.sunysb.edu}\\Dept. of Applied Mathematics \& Statistics\\
       Stony Brook University
\and
Hongyuan Zha\thanks{e-mail: zha@cc.gatech.edu}\\College of Computing\\
       Georgia Institute of Technology}

\keywords{normals, curvatures, principal directions, shape operators,
height function, stabilities}

\begin{document}

%\teaser{
%  \includegraphics[width=1.5in]{sample.eps}
%  \caption{Lookit! Lookit!}
%}

\maketitle
\begin{abstract}
Differential quantities, including normals, curvatures, principal
directions, and associated matrices, play a fundamental
role in geometric processing and physics-based modeling. Computing
these differential quantities {\it consistently} on surface meshes is
important and challenging, and some existing methods often produce
inconsistent results and require {\it ad hoc} fixes.  In this paper,
we show that the computation of the gradient and Hessian of a height
function provides the foundation for consistently
computing the differential quantities. We derive simple, {\it
explicit} formulas for the transformations between the first- and
second-order differential quantities (i.e., normal vector and principal
curvature tensor) of a smooth surface and the first- and second-order
derivatives (i.e., gradient and Hessian) of its corresponding height
function. We then investigate a general, flexible numerical framework
to estimate the derivatives of the height function based on local 
polynomial fittings formulated as weighted least squares approximations. 
We also propose an iterative fitting scheme to improve accuracy. 
This framework generalizes polynomial
fitting and addresses some of its accuracy and stability issues, as
demonstrated by our theoretical analysis as well as experimental
results.
\end{abstract}

\begin{CRcatlist}
  \CRcat{I.3.5}{Computer Graphics}{Computational Geometry and Object Modeling}{Algorithms, Design, Experimentation}
\end{CRcatlist}

%% The ``\keywordlist'' command prints out the keywords.
\keywordlist

%%%
\section{Introduction}

\copyrightspace

Computing normals and curvatures is a fundamental problem for many
geometric and numerical computations, including feature detection,
shape retrieval, shape registration or matching, surface fairing,
surface mesh adaptation or remeshing, front tracking and moving
meshes. In recent years, a number of methods have been introduced for
the computation of the differential quantities (see e.g.,
\cite{Taubin95ETC,MW00SNG,MDSB02DDG,CP05EDQ}). 
However, some of the existing methods may produce inconsistent
results.  For example, when estimating the mean curvature using the
cotangent formula and estimating the Gaussian curvature using the
angle deficit \cite{MDSB02DDG}, the principal curvatures obtained from
these mean and Gaussian curvatures are not guaranteed to be real
numbers. Such inconsistencies often require {\it ad hoc} fixes to
avoid crashing of the code, and their effects on the accuracy of the
applications are difficult to analyze.

%We say a solution is {\em consistent} if it is
%the exact solution of a perturbed problem. Using the terminology from
%numerical analysis, consistency is a necessary condition of {\em
%backward stability}, which require the solution to be the exact
%solution of a perturbed problem that is close to the input. Many 
%existing methods are not backward stable. 

The ultimate goal of this work is to investigate a mathematically
sound framework that can compute the differential quantities {\em
consistently} (i.e., satisfying the intrinsic constraints) with
provable {\em convergence} on general surface meshes, while being
flexible and easy to implement. This is undoubtly an ambitious
goal. Although we may have not fully achieved the goal, we make 
some contributions toward it. First, using the
singular value decomposition \cite{GV96MC} of the Jacobian matrix, we
derive explicit formulas for the transformations between the first-
and second-order differential quantities of a smooth surface (i.e.,
normal vector and principal curvature tensor) and the first- and second-order
derivatives of its corresponding height function (i.e., gradient and
Hessian). We also give the explicit formulas for the transformations
of the gradient and Hessian under a rotation of the coordinate
system. These transformations can be obtained without forming the
shape operator and the associated computation
of its eigenvalues or eigenvectors.  We then
reduce the problem, in a consistent manner, to the computation of the
gradient and Hessian of a function in two dimensions, which can then
be computed more readily by customizing classical numerical techniques.

Our second contribution is to develop a general and flexible method
for differentiating a height function, which can be viewed as a
generalization of the polynomial fitting of osculating jets
\cite{CP05EDQ}. Our method is based on the Taylor series
expansions of a function and its derivatives, then solved by a
weighted least squares formulation. We also propose an
iterative-fitting method that improves the accuracy of fittings.
Furthermore, we address the numerical stability issues by a systematic
point-selection strategy and a numerical solver with safeguard. We
present experimental results to demonstrate the accuracy and stability
of our approach for fittings up to degree six.

The remainder of this paper is organized as follows.
Section~\ref{sec:related_work} reviews some related work on the
computation of differential quantities of discrete surfaces.
Section~\ref{sec:continuous-surf} analyzes the stability and
consistency of classical formulas for computing differential
quantities for continuous surfaces and establishes the theoretical
foundation for consistent computations using a height function.
Section~\ref{sec:wlsf} presents a general
framework and a unified analysis for computing the gradient and
Hessian of a height function based on polynomial fittings, including
an iterative-fitting scheme. Section~\ref{sec:Experimental-Study}
presents some experimental results to demonstrate the accuracy and
stability of our approach. Section~\ref{sec:Conclusion} concludes the
paper with a discussion on future research directions.

\section{Related Work}
\label{sec:related_work}

Many methods have been proposed for computing or estimating the first- and
second-order differential quantities of a surface.  In recent years,
there have been significant interests in the convergence and
consistency of these methods.  We do not attempt to give a
comprehensive review of these methods but consider only a few of them
that are more relevant to our proposed approach; readers are referred to
\cite{Petitjean02SMR} and \cite{GG06ECT} for comprehensive surveys. 
Among the existing methods, many of them estimate the
different quantities separately of each other. For the estimation of
the normals, a common practice is to estimate vertex normals using a
weighted average of face normals, such as area weighting or angle
weighting. These methods in general are only first-order accurate,
although they are the most efficient.

For curvature estimation, Meyer et al. \shortcite{MDSB02DDG} proposed
some discrete operators for computing normals and curvatures, among
which the more popular one is the cotangent formula for mean
curvatures, which is closely related to the formula for Dirichlet
energy of Pinkall and Polthier \shortcite{PP93CDM}.
Xu  \shortcite{Xu04CDL}
 studied the convergence of a number of methods for estimating
the mean curvatures (and more generally, the Laplace-Beltrami
operators). It was concluded that the cotangent
formula does not converge except for some special cases, as also
noted in \cite{HPW05CMG,Ward05CCF}. 
Langer et al. \shortcite{LBS05ADD,LBS05EAQ} proposed a tangent-weighted formula for
estimating mean-curvature vectors, whose convergence also relies on special
symmetric patterns of a mesh.  Agam and
Tang \shortcite{AT05SFA} proposed a framework to estimate curvatures of
discrete surfaces based on local directional curve sampling.
Cohen-Steiner and Morvan \shortcite{CM03RDT} used normal cycle to analyze convergence of
curvature measures for densely sampled surfaces. 

Vertex-based quadratic or higher-order polynomial fittings can produce
convergent normal and curvature estimations.  Meek and Walton
\shortcite{MW00SNG} studied
the convergence properties of a number of estimators for normals and
curvatures.  It was further generalized to
higher-degree polynomial fittings by Cazals and Pouget \shortcite{CP05EDQ}. These
methods are most closely related to our approach. It is
well-known that these methods may encounter numerical difficulties at
low-valence vertices or special arrangements of vertices \cite{LS86CSF}, 
which we address in this paper. Razdan and
Bae \shortcite{RB05CES} proposed a scheme to estimate curvatures using 
biquadratic B\'ezier patches.  Some methods were also proposed to
improve robustness of curvature estimation under noise. 
For example, Ohtake et al. \shortcite{OBS04RVL} 
estimated curvatures by fitting
the surface implicitly with multi-level meshes. Recently, Yang et al.
\shortcite{YLHP06RPC} proposed to improve robustness
by adapting the neighborhood sizes. These methods in general only
provide curvature estimations that are meaningful in some average
sense but do not necessarily guarantee convergence of pointwise
estimates.

Some of the methods estimate the second-order differential quantities
from the surface normals. Goldfeather and Interrante \shortcite{GI04NCO} proposed a
cubic-order formula by fitting the positions and normals of the
surface simultaneously to estimate curvatures and principal
directions. Theisel et al. \shortcite{TRZS04NBE} proposed a face-based
approach for computing shape operators using linear interpolation of
normals. Rusinkiewicz  \shortcite{Rusink04ECT} proposed a similar face-based
curvature estimator from vertex normals. These
methods rely on good normal estimations for reliable results.  Zorin
and coworkers \cite{Zorin05CBE,GGRZ06CDS} proposed to compute a shape
operator using mid-edge normals, which resembles and ``corrects'' the
formula of \cite{OC93DTP}.  Good results were obtained in
practice but there was no theoretical guarantee of its order of
convergence.

\section{Consistent Formulas of Curvatures for Continuous Surfaces}
\label{sec:continuous-surf}

%Our framework
%belongs to the latter type, for its simplicity,
%rotation-and-translation invariance, and verifiable numerical
%stability, as we explain in this section.

In this section, we first consider the computation of differential
quantities of continuous surfaces (as opposed to discrete surfaces). This
is a classical subject of differential geometry (see e.g., \cite{doCarmo76DG}), 
but the differential quantities are traditionally expressed in terms
of the first and second fundamental forms, which are geometrically
intrinsic but sometimes difficult to understand when a change of
coordinate system is involved.  Furthermore, as we
will show shortly, it is not straightforward to evaluate the classical
formulas consistently (i.e., in a backward-stable fashion) in the
presence of round-off errors due to the potential loss of symmetry and
orthogonality. Starting from these classical formulas, we derive some
simple formulas by a novel use of the singular value
decomposition of the Jacobian of the height functions.
The formulas are easier to understand and to
evaluate. To the best of our knowledge, some of our formulas 
(including those for the symmetric shape operator and principal
curvature tensor) have not previously appeared in the literature.

\subsection{Classical Formulas for Height Functions}

We first review some well-known concepts and formulas in differential
geometry and geometric modeling. Given a smooth surface defined in the global $xyz$
coordinate system, it can be transformed into a local $uvw$ coordinate
system by translation and rotation. Assume both coordinate frames are
orthonormal right-hand systems. Let the origin of the local frame be
$\left[x_{0},y_{0},z_{0}\right]^T$ (note that  for convenience we 
treat points as column vectors). Let $\hat{\vec{t}}_{1}$ and
$\hat{\vec{t}}_{2}$ be the unit vectors in the global coordinate
system along the positive directions of the $u$ and $v$ axes,
respectively. Then,
$\hat{\vec{m}}=\hat{\vec{t}}_{1}\times\hat{\vec{t}}_{2}$ is the unit
vector along the positive $w$ direction. Let $\vec{Q}$ denote the
orthogonal matrix composed of column vectors $\hat{\vec{t}}_{1}$,
$\hat{\vec{t}}_{2}$, and $\hat{\vec{m}}$, i.e.,
\begin{equation}
\vec{Q}\equiv\left[\left.\begin{array}{c}
\\
\hat{\vec{t}}_{1}\\
\\
\end{array}\right|\begin{array}{c}
\\
\hat{\vec{t}}_{2}\\
\\
\end{array}\left|\begin{array}{c}
\\
\vec{m}\\
\\
\end{array}\right.\right],\end{equation}
where `$\mid$' denotes concatenation. 
Any point $\vec{x}$ on the surface is then transformed to a point
$\left[u,v,f(\vec{u})\right]^T\equiv\vec{Q}^{T}\left(\vec{x}-\vec{x}_{0}\right)$,
where $\vec{u}\equiv(u,v)$.
In general, $f$ is not a one-to-one mapping over the whole surface,
but if the $uv$ plane is close to the tangent plane at a point
$\vec{x}$, then $f$ would be one-to-one in a neighborhood of
$\vec{x}$. We refer to this function
$f(\vec{u}):\mathbb{R}^{2}\rightarrow\mathbb{R}$ (more precisely, from a
subset of $\mathbb{R}^{2}$ into $\mathbb{R}$) as a \emph{height
function} at $x$ in the $uvw$ coordinate frame. In the following, all
the formulas are given in the $uvw$ coordinate frame unless otherwise
stated. 

If the surface is smooth, the height function $f$ defines a smooth
surface composed of points $\vec{p}(\vec{u})=\left[u,v,f(\vec{u})\right]^T\in\mathbb{R}^{3}$.
Let $\nabla f\equiv\left[
\begin{array}{c}f_{u}\\ f_{v}\end{array}\right]$ denote the gradient of $f$
with respect to $\vec{u}$ and $\vec{H}\equiv\left[
\begin{array}{cc}
f_{uu} & f_{uv}\\
f_{vu} & f_{vv}\end{array}\right]$ the Hessian of
$f$, where $f_{uv}=f_{vu}$. 
The Jacobian of $\vec{p}(\vec{u})$ with respect to $\vec{u}$,
denoted by $\vec{J}$, is then \begin{equation}
\vec{J}\equiv\left[
\left.\begin{array}{c}\\
\vec{p}_{u} \\ \\ 
\end{array}\right|
\begin{array}{c}\\ 
\vec{p}_{v} \\ \\ 
\end{array}\right]=
\left[\begin{array}{cc}
1 & 0\\
0 & 1\\
f_{u} & f_{v}\end{array}\right].\label{eq:jacobian}\end{equation}
The vectors $\vec{p}_{u}$ and $\vec{p}_{v}$ form a
basis of the tangent space of the surface at $\vec{p}$, though they
may not be unit vectors and may not be orthogonal to each other. Let 
$d\vec{u}$ denote $\left[\begin{array}{c}du\\ dv\end{array}\right]$. The
\emph{first fundamental form} of the surface is given by the quadratic
form \[
I(d\vec{u})=d\vec{u}^{T}\vec{G}d\vec{u},\mbox{ where }\vec{G}\equiv\vec{J}^{T}\vec{J}=\left[\begin{array}{cc}
1+f_{u}^{2} & f_{u}f_{v}\\
f_{u}f_{v} & 1+f_{v}^{2}\end{array}\right].\]
$\vec{G}$ is called the {\em first fundamental matrix}.
Its determinant is $g\equiv\mbox{det}(\vec{G})=1+f_{u}^{2}+f_{v}^{2}$.
We introduce a symbol $\ell\equiv\sqrt{g}$, which will reoccur many
times throughout this section. Geometrically, 
$\ell=\Vert\vec{p}_{u}\times\vec{p}_{v}\Vert$, so it is the
``area element'' and is equal to the area of the parallelogram 
spanned by $\vec{p}_{u}$ and $\vec{p}_{v}$. 
The unit normal to the surface is then\begin{equation}
\hat{\vec{n}}=\frac{\vec{p}_{u}\times\vec{p}_{v}}{\ell}=\frac{1}{\ell}\left[\begin{array}{c}
-f_{u}\\
-f_{v}\\
1\end{array}\right].\label{eq:normal}\end{equation}
The normal in the global $xyz$ frame is then \begin{equation}
\hat{\vec{n}}_{g}=\frac{1}{\ell}\vec{Q}\left[\begin{array}{c}
-f_{u}\\
-f_{v}\\
1\end{array}\right]=\frac{1}{\ell}(\hat{\vec{m}}-f_{u}\hat{\vec{t}}_{1}-f_{v}\hat{\vec{t}}_{2}).\label{eq:normal_face}\end{equation}
The \emph{second fundamental form} in the basis $\{\vec{p}_{u},\vec{p}_{v}\}$
is given by the quadratic form \begin{equation*}
II(d\vec{u})=d\vec{u}^{T}\vec{B}d\vec{u}\end{equation*}
where 
\begin{equation}\vec{B}=-\left[\begin{array}{cc}
\hat{\vec{n}}_{u}^{T}\vec{p}_{u} & \hat{\vec{n}}_{u}^{T}\vec{p}_{v}\\
\hat{\vec{n}}_{v}^{T}\vec{p}_{u} & \hat{\vec{n}}_{v}^{T}\vec{p}_{v}\end{array}\right]=\left[\begin{array}{cc}
\hat{\vec{n}}^{T}\vec{p}_{uu} & \hat{\vec{n}}^{T}\vec{p}_{uv}\\
\hat{\vec{n}}^{T}\vec{p}_{uv} & \hat{\vec{n}}^{T}\vec{p}_{vv}\end{array}\right].\label{eq:secondff}\end{equation}
$\vec{B}$ is called the {\em second fundamental matrix},
and it measures the change of surface normals. It is easy to verify
that \begin{equation}\vec{B}=\vec{H}/\ell\end{equation} by plugging (\ref{eq:jacobian}) and
(\ref{eq:normal}) into (\ref{eq:secondff}).

In differential geometry, the well-known \emph{Weingarten equations}
read $\left[
\hat{\vec{n}}_{u} \mid \hat{\vec{n}}_{v} \right]=
-\left[\vec{p}_{u} \mid \vec{p}_{v}\right]\vec{W}$,
where $\vec{W}$
is the \emph{Weingarten matrix} (a $2\times2$ matrix a.k.a. the \emph{shape
operator}) at point $\vec{p}$ with basis $\{\vec{p}_{u},\vec{p}_{v}\}$.
By left-multiplying $\vec{J}^{T}$ on both sides, 
we have $\vec{B}=\vec{G}\vec{W}$, and therefore, \begin{equation}
\vec{W}=\vec{G}^{-1}\vec{B}=\frac{1}{\ell}(\vec{J}^{T}\vec{J})^{-1}\vec{H}.\label{eq:shapeop}\end{equation}
The eigenvalues $\kappa_{1}$ and $\kappa_{2}$ of the Weingarten matrix
$\vec{W}$ are the \emph{principal curvatures} at $\vec{p}$. The
\emph{mean curvature} is $\kappa_{H}=(\kappa_{1}+\kappa_{2})/2$
and is equal to half of the trace of $\vec{W}$. The \emph{Gaussian
curvature} at $\vec{p}$ is $\kappa_{G}=\kappa_{1}\kappa_{2}$ and
is equal to the determinant of $\vec{W}$. Note that 
$\kappa_{H}^{2}\geq\kappa_{G}$, because 
$4(\kappa_{H}^{2}-\kappa_{G})=(\kappa_{1}-\kappa_{2})^{2}$.
Let $\hat{\vec{d}}_{1}$ and $\hat{\vec{d}}_{2}$ denote the eigenvectors of $\vec{W}$.
Then $\hat{\vec{e}}_{1}=\vec{J}\hat{\vec{d}}_{1}/\Vert\vec{J}\hat{\vec{d}}_{1}\Vert$
and $\hat{\vec{e}}_{2}=\vec{J}\hat{\vec{d}}_{2}/\Vert\vec{J}\hat{\vec{d}}_{2}\Vert$
are the \emph{principal directions}.

The principal curvatures and principal directions are often used to
construct the $3\times3$ matrix \begin{equation}
\vec{C}\equiv\kappa_{1}\hat{\vec{e}}_{1}\hat{\vec{e}}_{1}^{T}+\kappa_{2}\hat{\vec{e}}_{2}\hat{\vec{e}}_{2}^{T}.\label{eq:curvtens1}\end{equation}
We refer to $\vec{C}$ as the \emph{principal curvature tensor} in the
local coordinate system. Note that $\vec{C}$ is sometimes simply
called the \emph{curvature tensor}, which is an overloaded term 
(see e.g., \cite{doCarmo92RG}), but we nevertheless 
will use it for conciseness in later discussions of this paper.
Note that $\hat{\vec{n}}$ must be in the null space of $\vec{C}$,
i.e., $\vec{C}\hat{\vec{n}}=\vec{0}$.  The curvature tensor
is particularly useful because it can facilitate coordinate
transformation. In particular, the curvature tensor in the
global coordinate system is $\vec{C}_{g}=\vec{Q}\vec{C}\vec{Q}^{T}$.

\subsection{Consistency of Classical Formulas}

The Weingarten equation is a standard way for defining and computing
the principal curvatures and principal directions in differential
geometry \cite{doCarmo76DG}, so it is important to understand its
stability and consistency. We notice that the singular-value
decomposition (SVD) of the Jacobian matrix $\vec{J}$ is given
by\begin{equation}
\vec{J}=\underbrace{\left[\begin{array}{cc}
c/\ell & -s\\
s/\ell & c\\
\Vert\nabla f\Vert/\ell & 0\end{array}\right]}_{\vec{U}}\underbrace{\left[\begin{array}{cc}
\ell & 0\\
0 & 1\end{array}\right]}_{\vec{\Sigma}}\underbrace{\left[\begin{array}{cc}
c & -s\\
s & c\end{array}\right]^{T}}_{\vec{V}^{T}},
\label{eq:jacob_svd}\end{equation}
with\begin{equation}
c=\cos\theta=\frac{f_{u}}{\Vert\nabla f\Vert},\mbox{ and } s=\sin\theta=\frac{f_{v}}{\Vert\nabla f\Vert},
\label{eq:theta}\end{equation}
where $\theta=\arctan(f_{v}/f_{u})$. 
In the special case of $f_{u}=f_{v}=0$, we have $\ell=1$, and any
$\theta$ leads to a valid SVD of $\vec{J}$. To resolve this
singularity, we define $\theta=0$ when $f_{u}=f_{v}=0$.

In (\ref{eq:jacob_svd}), the column vectors of $\vec{U}$ are the
left singular vectors of $\vec{J}$, which form an orthonormal basis
of the tangent space, the diagonal entries of $\vec{\Sigma}$ are
the singular values, and $\vec{V}$ is an orthogonal matrix composed
of the right singular vectors of $\vec{J}$. The condition number
of a matrix in 2-norm is defined as the ratio between its largest 
and the smallest singular values, so we obtain the following.
\begin{lem}
The condition number of $\vec{J}$ in 2-norm is
$\ell$ and the condition number of the first
fundamental matrix $\vec{G}$ is $\ell^{2}$.
\end{lem}
Note that the condition number of $\vec{G}$ is equal to its 
determinant. This is a
coincidence because one of the singular values of $\vec{J}$ and in
turn $\vec{G}$ happens to be equal to $1$. Suppose $\vec{H}$ has an
absolute error $\delta\vec{H}$ of $O(\epsilon\Vert\vec{H}\Vert_{2})$. From 
the standard perturbation analysis, we know that the relative 
error in $\vec{W}$ computed from (\ref{eq:shapeop}) can be as large 
as $O(\ell^{2}\epsilon)$. The mean
and Gaussian curvatures, if computed from the trace and determinant of
$\vec{W}$, would then have an absolute error of
$O(\ell^{2}\epsilon\Vert\vec{W}\Vert_{2})$. Therefore, the Weingarten
equation must be avoided if $\ell^2\gg 1$, and it is desirable to keep
$\ell$ as close to $1$ as possible. Note that $\vec{W}$ is
nonsymmetric, and more precisely it is nonnormal, i.e.,
$\vec{W}^{T}\vec{W}\neq\vec{W}\vec{W}^{T}$. The eigenvalues of
nonsymmetric matrices are not necessarily real, and the eigenvectors
of a nonnormal matrix are not orthogonal to each other \cite{GV96MC}. 
Therefore, in the presence of round-off or discretization errors, the
principal curvatures computed from the Weingarten matrix are not
guaranteed to be real, and the principal directions are not guaranteed
to be orthogonal, causing potential inconsistencies.

Unlike the Weingarten matrix $\vec{W}$, a curvature tensor is always
symmetric.  The curvature tensors are important concepts and are
widely used nowadays (see e.g., \cite{GG06ECT,TRZS04NBE}). In
principle, the two larger eigenvalues (in terms of magnitude) of a
curvature tensors corresponding to the principal curvatures and their
corresponding eigenvectors give a set of orthogonal principal
directions. As long as the estimated curvatures are real, one can use
(\ref{eq:curvtens1}) to construct a curvature tensor $\vec{C}$, except
that $\vec{C}$ would be polluted by the errors associated with the
nonorthogonal eigenvectors of $\vec{W}$. This pollution
can in turn lead to large errors in both the eigenvalues and
eigenvalues of $\vec{C}$. For example, for a spherical patch
$z=\sqrt{4-(x-0.5)^2-(y-0.5)^2}$ over the interval
$[0,1]\times[0,1]$, we observed a relative error of up to $8.5\%$ in
the eigenvalues of the curvature tensors constructed from a nonsymmetric
Weingarten matrix. Another serious problem associated with curvature
tensors is that if the minimum curvature is close to zero, then the
eigenvectors of a curvature tensor are extremely ill-conditioned.
Therefore, one must avoid computing the principal directions as the
eigenvectors of the curvature tensor if the minimum curvature
is close to 0.

\subsection{Simple, Consistent Formulas}
\label{sub:curvcomp}
To overcome the potential inconsistencies of the classical formulas of
principal curvatures and principal directions in the presence of
large $\ell$ or round-off errors, we derive a different set of
explicit formulas. Our main idea is to enforce orthogonality and
symmetry explicitly and to make the error propagation independent 
of $\ell$ as much as possible, starting from the following 
simple observation.

\begin{lem}
\label{lem:shapop_orth}Let $\vec{U}=\left[\vec{u}_{1}\mid\vec{u}_{2}\right]$ 
be a $3\times2$
matrix whose column vectors form an orthonormal basis of the Jacobian, 
and $\vec{C}$ be the curvature tensor. $\vec{U}^{T}\vec{C}\vec{U}$ is the
shape operator in the basis $\{\vec{u}_{1},\vec{u}_{2}\}$, 
and it is a symmetric matrix with an orthogonal diagonalization
$\vec{X}\vec{\Lambda}\vec{X}^{-1}$, whose eigenvalues (i.e., the
diagonal entries of $\vec{\Lambda}$) are real and whose eigenvectors
are orthonormal (i.e., $\vec{X}^T=\vec{X}^{-1}$). The column vectors of
$\vec{U}\vec{X}$ are the principal directions.
\end{lem}
\begin{proof}
Let $\kappa_{i}$ and $\hat{\vec{d}}_{i}$ be the eigenvalues and eigenvectors
of the Weingarten matrix $\vec{W}$ in (\ref{eq:shapeop}). Therefore,
\begin{equation}\vec{B}\hat{\vec{d}}_{i}=\vec{G}\vec{W}\hat{\vec{d}}_{i}=\kappa_{i}\vec{G}\hat{\vec{d}}_{i}.\label{eq:eigen}\end{equation}
Let $\vec{S}=\vec{U}^{T}\vec{J}$. By left-multiplying $\vec{S}^{-T}$
on both sides of (\ref{eq:eigen}), we have \[
\vec{S}^{-T}\vec{B}\vec{S}^{-1}(\vec{S}\hat{\vec{d}}_{i})=\kappa_{i}(\vec{S}^{-T}\vec{G}\hat{\vec{d}}_{i})=\kappa_{i}(\vec{S}\hat{\vec{d}}_{i}).\]
Therefore, the eigenvalues of $\vec{S}^{-T}\vec{B}\vec{S}^{-1}$
are the principal curvatures $\kappa_{i}$ and its eigenvectors
are the principal directions $\tilde{\vec{d}}_{1}\equiv\vec{S}\hat{\vec{d}}_{1}$
and $\tilde{\vec{d}}_{2}\equiv\vec{S}\hat{\vec{d}}_{2}$ in the basis 
$\{\vec{u}_{1},\vec{u}_{2}\}$, so $\vec{S}^{-T}\vec{B}\vec{S}^{-1}$ 
is the shape operator in this basis.
This shape operator is symmetric, so its eigenvalues
are real and its eigenvectors are orthonormal. By definition of the
curvature tensor, \begin{align*}
\vec{C} & = \kappa_{1}\vec{U}\tilde{\vec{d}}_{1}\tilde{\vec{d}}_{1}^{T}\vec{U}^{T}+\kappa_{2}\vec{U}\tilde{\vec{d}}_{2}\tilde{\vec{d}}_{2}^{T}\vec{U}^{T}\\
 & = \vec{U}(\kappa_{1}\tilde{\vec{d}}_{1}\tilde{\vec{d}}_{1}^{T}+\kappa_{2}\tilde{\vec{d}}_{2}\tilde{\vec{d}}_{2}^{T})\vec{U}^{T}\\
 & = \vec{U}\vec{S}^{-T}\vec{B}\vec{S}^{-1}\vec{U}^{T},\end{align*}
so $\vec{U}^{T}\vec{C}\vec{U}$ is equal to the shape operator $\vec{S}^{-T}\vec{B}\vec{S}^{-1}$. 
\end{proof}
Lemma~\ref{lem:shapop_orth} holds for any orthonormal basis of the
Jacobian. There is an infinite number of such bases. Algorithmically,
such a basis can be obtained from either the QR factorization or SVD
of $\vec{J}$. We choose the latter for its simplicity and elegance,
and it turns out to be particularly revealing.

\begin{thm}
\label{thm:shapeop_svd}Let $\vec{J}=\vec{U}\vec{\Sigma}\vec{V}^{T}$
be the reduced SVD of the Jacobian of a height function $f$ as given
by (\ref{eq:jacob_svd}). Let $\vec{u}_{1}$ and $\vec{u}_{2}$ denote
the column vectors of $\vec{U}$ and let $\vec{S}=\vec{\Sigma}\vec{V}^{T}$.
The shape operator in the orthonormal basis $\{\vec{u}_{1},\vec{u}_{2}\}$
is the symmetric matrix \begin{equation}
\tilde{\vec{W}}=\vec{S}^{-T}\vec{B}\vec{S}^{-1}=\frac{1}{\ell}\left[\begin{array}{cc}
c/\ell & s/\ell\\
-s & c\end{array}\right]\vec{H}\left[\begin{array}{cc}
c/\ell & -s\\
s/\ell & c\end{array}\right],\label{eq:shapeop_orth}\end{equation}
where $c$ and $s$ are defined in (\ref{eq:theta}).
\end{thm}
\begin{proof}
Note that \[
\vec{S}^{-1}=\vec{V}\vec{\Sigma}^{-1}=\left[\begin{array}{cc}
c & -s\\
s & c\end{array}\right]\left[\begin{array}{cc}
1/\ell & 0\\
0 & 1\end{array}\right]=\left[\begin{array}{cc}
c/\ell & -s\\
s/\ell & c\end{array}\right].\]
Following the same argument as for Lemma~\ref{lem:shapop_orth},
Eq.~(\ref{eq:shapeop_orth}) holds.
\end{proof}

To the best of our knowledge, the symmetric shape operator
$\tilde{\vec{W}}$ defined in (\ref{eq:shapeop_orth}) has not previously 
appeared in the literature. Because $\tilde{\vec{W}}$ is symmetric,
its eigenvalues (i.e., principal curvatures) are guaranteed to be
real, and its eigenvectors (i.e., principal directions) are guaranteed
to be orthonormal. In addition, because $\tilde{\vec{W}}$ is given
explicitly by (\ref{eq:shapeop_orth}), the relative error in $\vec{H}$
is no longer amplified by a factor of $\ell^2$ (or even $\ell$) in the
computation of $\tilde{\vec{W}}$. Therefore, this symmetric
shape operator provides a consistent way for computing the
principal curvatures or principal directions. Furthermore, using this
result we can now compute the curvature tensor even without forming
the shape operator.

\begin{thm}
\label{thm:The-curvature-tensors}Let the transformation from the local
to global coordinate system be $\vec{x}=\vec{Q}\vec{u}+\vec{x}_{0}$.
The curvature tensors in the local and global coordinate systems are
\begin{align}
\vec{C} & = \vec{J}^{+T}\vec{B}\vec{J}^{+}=\frac{1}{\ell}\vec{J}^{+T}\vec{H}\vec{J}^{+}\label{eq:curvtens_l}\\
\vec{C}_{g} & = \vec{J}_{g}^{+T}\vec{B}\vec{J}_{g}^{+}=\frac{1}{\ell}\vec{J}_{g}^{+T}\vec{H}\vec{J}_{g}^{+},\label{eq:curvtens_g}\end{align}
respectively, where \begin{equation}
\vec{J}^{+}=\frac{1}{\ell^{2}}\left[\begin{array}{ccc}
1+f_{v}^{2} & -f_{u}f_{v} & f_{u}\\
-f_{u}f_{v} & 1+f_{u}^{2} & f_{v}\end{array}\right]\label{eq:jacobi_inv}\end{equation}
is the pseudo-inverse of $\vec{J}$, and $\vec{J}_g^{+}=\vec{J}^{+}\vec{Q}^{T}$
is the pseudo-inverse of $\vec{J}_g=\vec{Q}\vec{J}$. In addition,
$\vec{H}=\ell\vec{J}^{T}\vec{C}\vec{J}=\ell\vec{J}_{g}^{T}\vec{C}_{g}\vec{J}_{g}$. 
\end{thm}
\begin{proof}
The principal directions are $\hat{\vec{e}}_{1}=\vec{U}\tilde{\vec{d}}_{1}$
and $\hat{\vec{e}}_{2}=\vec{U}\tilde{\vec{d}}_{2}$ in the local coordinate
system, where $\tilde{\vec{d}}_{i}$ are the eigenvectors of $\tilde{\vec{W}}$.
Therefore, \begin{align*}
\vec{C} & = \kappa_{1}\hat{\vec{e}}_{1}\hat{\vec{e}}_{1}^{T}+\kappa_{2}\hat{\vec{e}}_{2}\hat{\vec{e}}_{2}^{T}\\
 & = \vec{U}\tilde{\vec{W}}\vec{U}^{T}\\
 & = (\vec{U}\vec{\Sigma}^{-1}\vec{V}^{T})\vec{B}(\vec{V}\vec{\Sigma}^{-1}\vec{U}^{T})\\
 & = \vec{J}^{+T}\vec{B}\vec{J}^{+}.\end{align*}
By left multiplying $\ell\vec{J}^{T}$
and right multiplying $\vec{J}$ of both sides, we then have \[
\ell\vec{J}^{T}\vec{C}\vec{J}=\ell\vec{J}^{T}\vec{J}^{+T}\vec{B}\vec{J}^{+}\vec{J}=\ell\vec{B}=\vec{H}.\]
The curvature tensor in the global coordinate system is $\vec{C}_{g}=\vec{Q}\vec{C}\vec{Q}^{T}$,
which results in (\ref{eq:curvtens_g}) and $\vec{H}=\ell\vec{J}_{g}^{T}\vec{C}_{g}\vec{J}_{g}$.
To obtain the explicit formula for $\vec{J}^{+}$, simply plug in
$c$, $s$, and $\ell$ into \[
\vec{J}^{+}=\left[\begin{array}{cc}
c & -s\\
s & c\end{array}\right]\left[\begin{array}{cc}
1/\ell & 0\\
0 & 1\end{array}\right]\left[\begin{array}{ccc}
c/\ell & s/\ell & \Vert\nabla f\Vert/\ell\\
-s & c & 0\end{array}\right].\]
Alternatively, we may obtain $\vec{J}^{+}$ by expanding and simplifying
$(\vec{J}^{T}\vec{J})^{-1}\vec{J}^{T}$. 
\end{proof}
%The above theorem may be surprising, because one might have expected
%that $\vec{J}^{T}\vec{C}\vec{J}$ should be $\vec{G}^{-1}\vec{H}/\ell$
%instead of $\vec{H}/\ell$ according to Lemma~\ref{lem:shapop_orth}.
%There is no contradiction here because Lemma~\ref{lem:shapop_orth}
%does not apply to $\vec{J}^{T}\vec{C}\vec{J}$ unless $\vec{J}$
%is orthogonal, in which case $\vec{G}=\vec{I}$ and $\vec{G}^{-1}\vec{H}=\vec{H}$.
Theorem~\ref{thm:The-curvature-tensors} allows us to transform between
the curvature tensor and the Hessian of a height function, both of
which are symmetric matrices. Furthermore, we can easily compute the
gradient and Hessian of one height function from their corresponding
values of another height function in a different coordinate system 
corresponding to the same point of a smooth surface.

\begin{cor}
\label{cor:gradient}Given the gradient $\nabla f=\left[
\begin{array}{c}f_{u}\\ f_{v}\end{array}\right]$ and
Hessian $\vec{H}$ of a height function $f$ in an orthonormal
coordinate system, let $\ell$ denote $\sqrt{1+f_{u}^{2}+f_{v}^{2}}$,
$\hat{\vec{n}}$ the unit normal
$\left[-f_{u},-f_{v},1\right]^{T}/\ell$, $\vec{C}$ the curvature
tensor $\vec{J}^{+T}\vec{H}\vec{J}^{+}/\ell$, where
$\vec{J}=\left[\vec{I}_{2\times2}\mid\nabla f\right]^{T}$.  Let
$\hat{\vec{Q}}\equiv\left[\hat{\vec{q}}_{1}\mid\hat{\vec{q}}_{2}\mid\hat{\vec{q}}_{3}\right]$
be the orthogonal matrix whose column vectors form the axes of another
coordinate system, where $\hat{\vec{q}}_{3}^{T}\hat{\vec{n}}>0$.  Let
$\left[\alpha,\beta,\gamma\right]^T$ denote
$\hat{\vec{Q}}^{T}\hat{\vec{n}}$, and let $\tilde{f}$ denote
the height function in the latter coordinate frame corresponding 
to the same smooth surface. Then the gradient of $\tilde{f}$ is
$\nabla\tilde{f}=\left[
\begin{array}{c} -\alpha/\gamma\\ -\beta/\gamma\end{array}\right]$, and
the Hessian of $\tilde{f}$ is
$\tilde{\vec{J}}^{T}\vec{C}\tilde{\vec{J}}/\gamma$, where
$\tilde{\vec{J}}=\hat{\vec{Q}}\left[\vec{I}_{2\times2}\mid\nabla\tilde{f}\right]^{T}$.
\end{cor}
A direct implication of this corollary is that knowing the
\textit{consistent} surface normal and curvature
tensor of a smooth surface in the global coordinate frame is
equivalent to knowing the \textit{consistent} gradient and
Hessian of its corresponding height function in any local
coordinate frame in which the Jacobian is nondegenerate (i.e., $\ell>0$). 
Here, the consistency refers to the fact that the normal $\hat{\vec{n}}$ is
in the null space of the curvature tensor $\vec{C}$ and that 
both $\vec{C}$ and the Hessian $\vec{H}$ are
symmetric (i.e., $\vec{C}\hat{\vec{n}}=\vec{0}$,
$\vec{C}^{T}=\vec{C}$, and $\vec{H}=\vec{H}^{T}$).

The preceding formulas for the symmetric shape operator and
curvature tensor appear to be new and are particularly useful when
computing the principal curvatures and principal directions. Because
many applications require the mean curvature and Gaussian curvature,
we also give the following simple formulas, which are equivalent
to those in classical differential geometry \cite[p. 163]{doCarmo76DG}
but are given here in a more concise form.

\begin{thm}
\label{thm:curvatures}The mean and Gaussian curvature of the height
function $f(\vec{u}):\mathbb{R}^{2}\rightarrow\mathbb{R}$ are \begin{equation}
\kappa_{H}=\frac{\mbox{tr}(\vec{H})}{2\ell}-\frac{(\nabla f)^{T}\vec{H}(\nabla f)}{2\ell^{3}},\mbox{ and }\kappa_{G}=\frac{\mbox{det}(\vec{H})}{\ell^{4}}.\label{eq:curvatures_stable}\end{equation}

\end{thm}
\begin{proof}
Let $\vec{v}=\left[c,s\right]^T$ and $\vec{v}^{\bot}=\left[-s,c\right]^T$, where $c$
and $s$ are defined in (\ref{eq:theta}). The trace of $\tilde{\vec{W}}$
is \begin{align*}
\mbox{tr}(\tilde{\vec{W}}) & = \mbox{tr}\left(\frac{1}{\ell}\left[\begin{array}{cc}
c/\ell & s/\ell\\
-s & c\end{array}\right]\vec{H}\left[\begin{array}{cc}
c/\ell & -s\\
s/\ell & c\end{array}\right]\right)\\
 & = \frac{1}{\ell}\left(\frac{\vec{v}^{T}}{\ell}\vec{H}\frac{\vec{v}}{\ell}+\vec{v}^{\bot T}\vec{H}\vec{v}^{\bot}\right)\\
 & = \frac{1}{\ell^{3}}\left(\vec{v}^{T}\vec{H}\vec{v}+\ell^{2}\vec{v}^{\bot T}\vec{H}\vec{v}^{\bot}\right)\\
 & = \frac{1}{\ell^{3}}\left(\ell^{2}\mbox{tr}(\vec{H})-(\ell^{2}-1)\vec{v}^{T}\vec{H}\vec{v}\right),\end{align*}
where the last step uses $\vec{v}^{T}\vec{H}\vec{v}+\vec{v}^{\bot T}\vec{H}\vec{v}^{\bot}=\mbox{tr}(\vec{H})$.
Since $\sqrt{\ell^{2}-1}\vec{v}=\nabla f$, therefore, \[
\kappa_{H}=\frac{1}{2}\mbox{tr}(\tilde{\vec{W}})=\frac{\mbox{tr}(\vec{H})}{2\ell}-\frac{(\nabla f)^{T}\vec{H}(\nabla f)}{2\ell^{3}}.\]
With regard to $\kappa_G$, we have \[
\kappa_{G}=\mbox{det}(\tilde{\vec{W}})=\mbox{det}(\vec{S}^{-1})^{2}\mbox{det}(\vec{H}/\ell)=\frac{\mbox{det}(\vec{H})}{\ell^{4}}.\]
\end{proof}

This completes our description of the explicit formulas for the
differential quantities. Since their derivations, especially for the
principal curvatures, principal directions, and curvature tensor, are
based on the shape operator associated with the left singular-vectors
of the Jacobian, many intermediate terms cancel out due to
symmetry and orthogonality, and the final formulas are remarkably
simple. It is important to note that even if the gradient or Hessian
contain input errors, as long as the Hessian is symmetric and the
Jacobian is nondegenerate, our formulas are consistent in
the following sense: The principal curvatures are real, the principal
directions are orthonormal, and the surface normal is orthogonal to
the column space of $\vec{J}$, the row spaces of $\vec{J}^{+}$, and
the column space of the curvature tensor.

The consistency of our formulas is enabled by a
simple fact: A consistent set of first- and second-order differential
quantities have five degrees of freedom, which is a direct result
of Corollary~\ref{cor:gradient}. In the Weingarten equations,
numerically there are six degrees of freedom (three in $\vec{G}$ and
three in $\vec{B}$), so there is an implicit constraint (which should
have enforced the orthogonality of the principal directions) not
present in the equations. In general, if a system of equations has
more degrees of freedom than the intrinsic dimension of the problem
(i.e., with implicit constraints), its numerical solution would likely lead to
inconsistencies in the presence of round-off errors.  The same
argument may be applied to other methods that assume more than
five independent parameters per point for the first- and second-order
differential quantities. For example, if the differential quantities
are evaluated from a parameterization of a surface where $x$, $y$, and
$z$ coordinates are viewed as independent functions of $u$ and $v$,
one must compute more than five parameters, so their numerical
solutions may not be consistent.  The gradient and the Hessian of the
height function contain exactly five degrees of freedom when the
symmetry of the Hessian is enforced, so we have reduced the problem
consistently into the computation of the gradient and Hessian of the
height function.  We therefore build our method based on the
estimation of the gradient and Hessian of the height functions.

\section{Computing Gradient and Hessian of Height Function}
\label{sec:wlsf}

To apply the formulas for continuous surfaces to discrete surfaces,
we must select a region of interest, define a local $uvw$ coordinate
frame, and approximate the gradient and Hessian of the resulting height
function. We build our method based on a local polynomial fitting.
Polynomial fitting is not a new idea and has been studied intensively
(e.g., \cite{LS86CSF,deBoor92CAM,MW00SNG,CP05EDQ}). However, it is
well-known that polynomial fitting may suffer from numerical instabilities,
which in turn can undermine convergence and lead to large errors.
We propose some techniques to overcome instabilities and to improve
the accuracy of fittings. These techniques are based on classical
concepts in numerical linear algebra but are customized here for
derivative computations. In addition, we propose a new
iterative-fitting scheme and also present a unified error analysis of
our approach.

\subsection{Local Polynomial Fitting}

\label{sub:local_poly}

The local polynomial fitting can be derived from the Taylor series
expansion. Let $f(\vec{u})$ denote a bivariate function, where $\vec{u}=(u,v)$,
and let $c_{jk}$ be a shorthand for $\frac{\partial^{j+k}}{\partial u^{j}\partial v^{k}}f(\vec{0})$.
The Taylor series expansion of $f$ about the origin $\vec{u}_{0}=(0,0)$
is\begin{equation}
f(\vec{u})=\sum_{p=0}^{\infty}\sum_{j,k\geq0}^{j+k=p}c_{jk}\frac{u^{j}v^{k}}{j!k!}.\label{eq:taylor_series_inf}\end{equation}
Given a positive integer $d$, a function $f(\vec{u})$ can be approximated 
to $(d+1)$st order accuracy about the origin $\vec{u}_{0}$ as\begin{equation}
f(\vec{u})=\underbrace{\sum_{p=0}^{d}\sum_{j,k\geq0}^{j+k=p}c_{jk}\frac{u^{j}v^{k}}{j!k!}}_{\mbox{Taylor polynomial}}+\underbrace{O(\Vert\vec{u}\Vert^{d+1})}_{\mbox{remainder}},\label{eq:taylor_series}\end{equation}
assuming $f$ has $d+1$ continuous derivatives. The derivatives of
the Taylor polynomial are the same as $f$ at $\vec{u}_{0}$ up to
degree $d$.

Given a set of points sampling a small patch of a smooth surface,
the method of local polynomial fitting approximates the Taylor polynomial
by estimating $c_{jk}$ from the given points. The degree of the polynomial,
denoted by $d$, is called the \emph{degree of fitting}. We will limit
ourselves to relatively low degree fittings (say $d\le6$), because
high-degree fittings tend to be more oscillatory and less stable. To estimate
$c_{jk}$ at a vertex of a surface mesh, we choose the vertex to be
the origin of a local $uvw$ coordinate frame, where the $w$ component
of the coordinates in this frame would be the height function $f$.
We use an approximate vertex normal (e.g., obtained by averaging the
face normals) as the $w$ direction, so that the condition number
of the Jacobian of $f$ would be close to 1. Plugging in each given
point $\left[u_{i},v_{i},f_{i}\right]^T$ into (\ref{eq:taylor_series}), we obtain
an approximate equation \begin{equation}
\sum_{p=0}^{d}\sum_{j,k\geq0}^{j+k=p}c_{jk}\frac{u_{i}^{j}v_{i}^{k}}{j!k!}\approx f_{i},\label{eq:linearsys}\end{equation}
which has $n\equiv(d+1)(d+2)/2$ unknowns (i.e., $c_{jk}$ for $0\leq j+k\leq d$,
$j\geq0$ and $k\geq0$). As a concrete example, for cubic fitting
the equation is\begin{align}
c_{00}+c_{10}u_{i}+c_{01}v_{i}+c_{20}\frac{u_{i}^{2}}{2}+c_{11}u_{i}v_{i}+\nonumber \\
c_{02}\frac{v_{i}^{2}}{2}+c_{30}\frac{u_{i}^{3}}{6}+c_{21}\frac{u_{i}^{2}v_{i}}{2}+c_{12}\frac{u_{i}v_{i}^{2}}{2}+c_{03}\frac{v_{i}^{3}}{6} & \approx f_i.\label{eq:linearsys_cubic}\end{align}
Let $m$ denote the number of these given points (including the vertex
$\vec{u}_{0}$ itself), we then obtain an $m\times n$ rectangular
linear system. If we want to enforce the fit to pass through the vertex
itself, we can simply set $c_{00}=0$ and remove the equation corresponding
to $\vec{u}_{0}$, leading to an $(m-1)\times(n-1)$ rectangular linear
system. In our experiments, it seems to have little benefit to do
this, so we do not enforce $c_{00}=0$ in the following discussions.
However, our framework can be easily adapted to enforce $c_{00}=0$
if desired.

Suppose we have solved this rectangular linear system and obtained
the approximation of $c_{jk}$. At $\vec{u}_{0}$, the gradient of
$f$ is $\left[\begin{array}{c}c_{10}\\c_{01}\end{array}\right]$ and the Hessian is $\left[\begin{array}{cc}
c_{20} & c_{11}\\
c_{11} & c_{02}\end{array}\right]$. Plugging their approximations into the formulas in Section~\ref{sec:continuous-surf}
would give us the normal and curvature approximations at $\vec{u}_{0}$. This approach
is similar to the fitting methods in \cite{CP05EDQ,MW00SNG}. Note
that the Taylor series expansions of $f_{u}$ and $f_{v}$ about $\vec{u}_{0}$
are\begin{align}
f_{u}(\vec{u}) & \approx\sum_{p=0}^{d-1}\sum_{j,k\geq0}^{j+k=p}c_{(j+1)k}\frac{u^{j}v^{k}}{j!k!},\label{eq:interp_gradu}\\
f_{v}(\vec{u}) & \approx\sum_{p=0}^{d-1}\sum_{j,k\geq0}^{j+k=p}c_{j(k+1)}\frac{u^{j}v^{k}}{j!k!},\label{eq:interp_gradv}\end{align}
where the residual is $O(\Vert\vec{u}\Vert^{d})$. The expansions
for the second derivatives have similar patterns. Plugging in the
approximations of $c_{jk}$ into these series, we can also obtain
the estimations of the gradient, Hessian, and in turn the curvatures
at a point $\vec{u}$ near $\vec{u}_{0}$. The remaining questions
for this approach are the theoretical issue of solving the rectangular
system as well as the practical issues of the selection of points
and robust implementation.

\subsection{Weighted Least Squares Formulation}

Let us denote the rectangular linear system obtained from (\ref{eq:linearsys})
for $i=1,\dots,m$ as \begin{equation}
\vec{V}\vec{c}\approx\vec{f},\label{eq:least_squares}\end{equation}
where $\vec{c}$ is an $n$-vector composed of $c_{jk}$, $\vec{f}$
is an $m$-vector composed of $f_i$, and $\vec{V}$ is
a \emph{generalized Vandermonde matrix}. For now, let us assume that
$m\geq n$ and $\vec{V}$ has full rank (i.e., its column vectors
are linearly independent). Such a system is said to be \emph{over-determined}.
We will address the robust numerical solutions for more general cases
in the next subsection.

The simplest solution to (\ref{eq:least_squares}) is to minimize
the 2-norm of the residual $\vec{V}\vec{c}-\vec{f}$, i.e., $\min_{\vec{c}}\Vert\vec{V}\vec{c}-\vec{f}\Vert_{2}$.
A more general formulation is to minimize a weighted norm (or semi-norm),
i.e., \begin{equation}
\min_{\vec{c}}\Vert\vec{V}\vec{c}-\vec{f}\Vert_{\vec{W}}\equiv\min_{\vec{c}}\Vert\vec{W}(\vec{V}\vec{c}-\vec{f})\Vert_{2},\label{eq:generalized_norm}\end{equation}
where $\vec{W}$ is an $m\times m$ diagonal matrix with nonnegative
entries. We refer to $\vec{W}$ as a \emph{weighting matrix}. 
Such a formulation is called \emph{weighted least squares}
\cite[p. 265]{GV96MC}, and it has a unique solution if $\vec{W}\vec{V}$
has full rank. It is equivalent to a linear least squares problem
\begin{equation}
\vec{A}\vec{c}\approx\vec{b},\mbox{ where }\vec{A}\equiv\vec{W}\vec{V}\mbox{ and }\vec{b}\equiv\vec{W}\vec{f}.\label{eq:weighted_least_squares}\end{equation}
Let $\omega_{i}$
denote the $i$th diagonal entry of $\vec{W}$. Algebraically, $\omega_{i}$
assigns a weight to each row of the linear system. Geometrically,
it assigns a priority to each point (the larger $\omega_{i}$, the
higher the priority). If $\vec{f}$ is in the column space of $\vec{V}$,
then a nonsingular weighting matrix has no effect on the solution
of the linear system. Otherwise, different weighting matrices may
lead to different solutions. Furthermore, by setting $\omega_{i}$
to be zero or close to zero, the weighting matrix can be used as a
mechanism for filtering out outliers in the given points.

The weighted least squares formulation is a general technique, but
the choice of $\vec{W}$ is problem dependent. For local polynomial
fitting, it is natural to assign lower priorities to points that are
farther away from the origin or whose normals differ substantially from
the $w$ direction of the local coordinate frame. Furthermore, the
choice of $\vec{W}$ may affect the residual and also the condition
number of $\vec{A}$. Let $\hat{\vec{m}}_{i}$ denote an initial approximation
of the unit normal at the $i$th vertex (e.g., obtained by averaging
face normals). Based on the above considerations, we choose the weight
of the $i$th vertex as\begin{equation}
\omega_{i}=\gamma_{i}^{+}\left/\left(\sqrt{\Vert\vec{u}_{i}\Vert^{2}+\epsilon}\right)^{d/2}\right.,\label{eq:weights}\end{equation}
where $\gamma_{i}^{+}\equiv\max(0,\hat{\vec{m}}_{i}^{T}\hat{\vec{m}}_{0})\mbox{ and }\epsilon\equiv\frac{1}{100m}\sum_{i=1}^{m}\Vert\vec{u}_{i}\Vert^{2}$.
In general, this weight is approximately equal to $\Vert\vec{u}_{i}\Vert^{-d/2}$,
where the exponent $d/2$ tries to balance between accuracy and stability.
%(because a large exponent (close to\emph{ $d$}) tends to equalize
%the residual but also increase the condition number of $\vec{A}$).
The factor $\gamma_{i}^{+}$ approaches $1$ for fine meshes at smooth
areas, but it serves as a safeguard against drastically changing normals
for coarse meshes or nonsmooth areas. The term $\epsilon$ prevents
the weights from becoming too large at points that are too close 
to $\vec{u}_{0}$.

The weighting matrix scales the rows of $\vec{V}$. However, if $u_{i}$
or $v_{i}$ are close to zero, the columns of $\vec{V}$ can be poorly
scaled, so that the $i$th row of $\vec{A}$ would be close to $[\omega_{i},0,\dots,0]$.
Such a matrix would have a very large condition number. A general
approach to alleviate this problem is to introduce a \emph{column
scaling matrix} $\vec{S}$. The least squares problem then becomes\begin{equation}
\min_{\vec{d}}\Vert(\vec{A}\vec{S})\vec{d}-\vec{b}\Vert_{2},\mbox{ where }\vec{d}\equiv\vec{S}^{-1}\vec{c}.\label{eq:weighted_lsf}\end{equation}
Here, $\vec{S}$ is a nonsingular $n\times n$ diagonal matrix. Unlike
the weighting matrix $\vec{W}$, the scaling matrix $\vec{S}$ does
not change the solution of $\vec{c}$ under exact arithmetic. However,
it can significantly improve the condition number and in turn improve
the accuracy in the presence of round-off errors. In general, let
$\vec{a}_{j}$ denote the $j$th column vector of $\vec{A}$. We choose
$\vec{S}=\mbox{diag}(1/\Vert\vec{a}_{1}\Vert_{2},\dots,1/\Vert\vec{a}_{n}\Vert_{2})$,
as it approximately minimizes the 2-norm condition number of $\vec{A}\vec{S}$
(see \cite[p. 265]{GV96MC} and \cite{VDS69CNE}).

\subsection{Robust Implementation}

\label{sub:solution}

Our preceding discussions considered only over-determined systems.
However, even if $m\geq n$, the matrix $\vec{V}$ (and in turn $\vec{A}\vec{S}$)
may not have full rank for certain arrangements of points, so the
weighted least squares would be practically under-determined with
an infinite number of solutions. Numerically, even if $\vec{V}$ has
full rank, the condition number of $\vec{A}\vec{S}$ may still be
arbitrarily large, which in turn may lead to arbitrarily large errors
in the estimated derivatives. We address these issues by proposing
a systematic way for selecting the points and a safeguard for 
QR factorization for solving the least squares problem.

\subsubsection{Section of Points}

As pointed out in \cite{LS86CSF}, the condition number of polynomial
fitting can depend on the arrangements of points. The situation may
appear to be hopeless, because we in general have little control (if
any at all) about the locations of the points. However, the problem
can be much alleviated when the number of points $m$ is substantially
larger than the number of unknowns $n$. To the best of our knowledge,
no precise relationship between the condition number
and $m/n$ has been established. However, it suffices here to have an intuitive understanding
from the geometric interpretation of condition numbers. Observe that
an $m\times n$ matrix $\vec{M}$ (with $m\geq n$) is rank deficient
if its row vectors are co-planar (i.e, contained in a hyperplane
of dimension less than $n$), and its condition number is very large
if its row vectors are nearly co-planar (i.e., if they lie
in the proximity of a hyperplane). Suppose the row vectors of $\vec{M}$ are random
with a nearly uniform distribution, then the probability of the vectors
being nearly co-planar decreases exponentially as $m$ increases.
Therefore, having more points than the unknowns can be highly beneficial
in decreasing the probability of ill-conditioning for well-scaled
matrices. In addition, it is well-known that having more points also
improves noise resistance of the fitting, which however is beyond
the scope of this paper.

Based on the above observation, we require $m$ to be larger than $n$,
but a too large $m$ would compromise efficiency and may also undermine
accuracy. As a rule of thumb,
it appears to be ideal for $m$ to be between $1.5n$ and $2n$. For
triangular meshes, we define the following neighborhoods that 
meet this criterion. First, let us define the \emph{1-ring} \emph{neighbor
faces} of a vertex to be the triangles incident on it. We define the
\emph{1-ring neighborhood }of a vertex to be the vertices of its 1-ring\emph{
}neighbor faces,\emph{ }and define the \emph{1.5-ring neighborhood
}as the vertices of all the faces that share an edge with a 1-ring
neighbor face. This definition of 1-ring neighborhood is conventional,
but our definition of 1.5-ring neighborhood appears to be new. Furthermore,
for $k\geq1$, we define the \emph{$(k+1)$-ring neighborhood} of
a vertex to be the vertices of the \emph{$k$-}ring neighborhood along
with their 1-ring neighbors, and define the \emph{$(k+1.5)$-ring
neighborhood} to be the vertices of the \emph{$k$-}ring neighborhood
along with their 1.5-ring neighbors. Figure~\ref{fig:neighbors}
illustrates these neighborhood definitions up to $2.5$ rings.

\begin{figure}
\begin{centering}
\includegraphics[width=0.9\columnwidth]{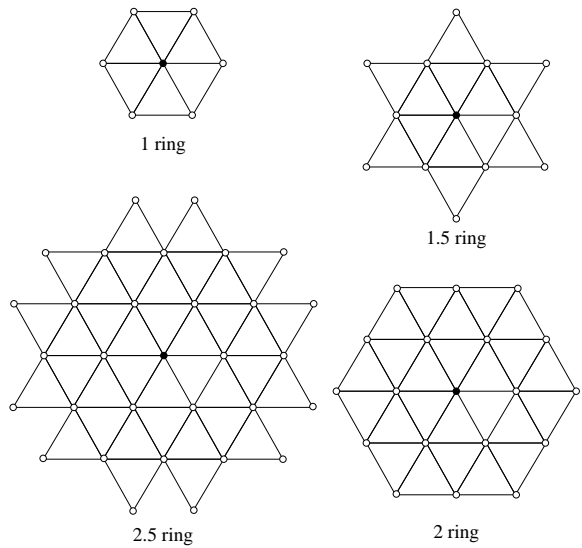}
\par\end{centering}

\caption{\label{fig:neighbors}Schematics of 1-ring, 1.5-ring, 2-ring, and
2.5-ring neighbor vertices of center vertex (black) in each illustration.}

\end{figure}

Observe that a typical $\frac{d+1}{2}$-ring neighborhood has about
the ideal number of points for the $d$th degree fittings at least up 
to degree six, which is evident from Table~\ref{tab:stencils}. Therefore, we
use this as the general guideline for selecting the points. Note that
occasionally (such as for the points near boundary) the 
$\frac{d+1}{2}$-ring neighborhood may not have enough points.
If the number of points is less than $1.5n$, we increase 
the ring level by $0.5$ up to $3.5$. 
Note that we do not attempt to filter out the points that
are far away from the origin at this stage, because such a filtering
is done more systematically through the weighting matrix in (\ref{eq:generalized_norm}).

\begin{table}
\begin{centering}
\caption{\label{tab:stencils}Numbers of coefficients in $d$th degree fittings
versus numbers of points in typical $\frac{d+1}{2}$ rings.}

\par\end{centering}

\begin{centering}
\begin{tabular}{c|c|c|c|c|c|c}
degree ($d$) & 1 & 2 & 3 & 4 & 5 & 6\tabularnewline
\hline
\hline 
\#coeffs. & 3 & 6 & 10 & 15 & 21 & 28\tabularnewline
\hline 
\#points in $\frac{d+1}{2}$ ring & 7 & 13 & 19 & 31 & 37 & 55\tabularnewline
\end{tabular}
\par\end{centering}
\end{table}

\subsubsection{QR Factorization with Safeguard}

Our discussions about the point selection was based on a probability
argument. Therefore, ill-conditioning may still occur especially near
the boundary of a surface mesh. The standard approaches for addressing
ill-conditioned rectangular linear systems include SVD or QR factorization
with column pivoting (see \cite[p. 270]{GV96MC}. When applied to
(\ref{eq:weighted_lsf}), the former approach would seek a solution
that minimizes the 2-norm of the solution vector $\Vert\vec{d}\Vert_{2}$
among all the feasible solutions, and the latter provides an efficient
but less robust approximation to the former. Neither of these approaches
seems appropriate in this context, because they do not give higher
priorities to the lower derivatives, which are the solutions
of interest.

Instead, we propose a safeguard for QR factorization for the local
fittings.  Let the columns of $\vec{V}$ (and in turn of $\vec{A}$) be
sorted in increasing order of the derivatives (i.e., in increasing
order of $j+k$).  Let the reduced QR factorization of $\vec{A}\vec{S}$
be \[\vec{A}\vec{S}=\vec{Q}\vec{R},\]
where $\vec{Q}$ is $m\times n$ with orthonormal column vectors and
$\vec{R}$ is an $n\times n$ upper-triangular matrix. The 2-norm
condition number of $\vec{A}\vec{S}$ is the same as that of $\vec{R}$.
To determine whether $\vec{A}\vec{S}$ is nearly rank deficient, we estimate
the condition number of $\vec{R}$ in 1-norm (instead of 2-norm, 
for better efficiency), which can be done efficiently for triangular
matrices and is readily available
in linear algebra libraries (such as DGECON in LAPACK \cite{ABB99LUG}).
If the condition number of $\vec{R}$ is too large (e.g., $\ge10^{3}$),
we decrease the degree of the fitting by removing the last few columns
of $\vec{A}\vec{S}$ that correspond to the highest derivatives. Note
that QR factorization need not be recomputed after decreasing
the degree of fitting, because it can be obtained by removing the
corresponding columns in $\vec{Q}$ and removing the corresponding
rows and columns in $\vec{R}$. If the condition number is still large,
we would further reduce the degree of fitting until the condition
number is small or the fitting becomes linear. Let $\tilde{\vec{Q}}$
and $\tilde{\vec{R}}$ denote the reduced matrices of $\vec{Q}$ and
$\vec{R}$, and the final solution of $\vec{c}$ is given by \begin{equation}
\vec{c}=\vec{S}\tilde{\vec{R}}^{-1}\tilde{\vec{Q}}^{T}\vec{b},\label{eq:solution}\end{equation}
 where $\tilde{\vec{R}}^{-1}$ denotes a back substitution step. Compared
to SVD or QR with partial pivoting, the above procedure is more accurate
for derivative estimation, as it gives higher priorities to lower
derivatives, and at the same time it is more efficient than SVD.

\subsection{Iterative Fitting of Derivatives}

\label{sec:fitting-derivatives}

The polynomial fitting above uses only the coordinates of the given
points. If accurate normals are known, it can be beneficial to take
advantage of them. If the normals are not given a priori, we may estimate
them first by using the polynomial fitting described earlier. We refer
to this approach as \emph{iterative fitting}.

First, we convert the vertex normals into the gradients of the height
function within the local $uvw$ frame at a point. Let $\hat{\vec{n}}_{i}=\left[\alpha_{i},\beta_{i},\gamma_{i}\right]^T$
denote the unit normal at the $i$th point in the $uvw$ coordinate
system, and then the gradient of the height function is $\left[
\begin{array}{c} -\alpha_i/\gamma_i\\ -\beta_i/\gamma_i\end{array}\right]$.
Let $a_{jk}\equiv c_{(j+1)k}$ and $b_{jk}\equiv c_{j(k+1)}$. By
plugging $u_{i}$, $v_{i}$, and $f_{u}(u_{i},v_{i})\approx-\alpha_{i}/\gamma_{i}$
into (\ref{eq:interp_gradu}), we obtain an equation for the coefficients
$a$s; similarly for $f_{v}$ and $b$s. For example, for cubic fittings
we obtain the equations
\begin{align}
a_{00}+a_{10}u_{i}+a_{01}v_{i}+a_{20}\frac{u_{i}^{2}}{2}+a_{11}u_{i}v_{i}+\nonumber \\
a_{02}\frac{v_{i}^{2}}{2}+a_{30}\frac{u_{i}^{3}}{6}+a_{21}\frac{u_{i}^{2}v_{i}}{2}+a_{12}\frac{u_{i}v_{i}^{2}}{2}+a_{03}\frac{v_{i}^{3}}{6} & \approx-\frac{\alpha_{i}}{\gamma_{i}},\\
b_{00}+b_{10}u_{i}+b_{01}v_{i}+b_{20}\frac{u_{i}^{2}}{2}+b_{11}u_{i}v_{i}+\nonumber \\
b_{02}\frac{v_{i}^{2}}{2}+b_{30}\frac{u_{i}^{3}}{6}+b_{21}\frac{u_{i}^{2}v_{i}}{2}+b_{12}\frac{u_{i}v_{i}^{2}}{2}+b_{03}\frac{v_{i}^{3}}{6} & \approx-\frac{\beta_{i}}{\gamma_{i}}.\end{align}
If we enforce $a_{j,k+1}=b_{j+1,k}$ explicitly, we would obtain a
single linear system for all the $a$s and $b$s with a reduced number
of unknowns. Alternatively, $a_{jk}$ and $b_{jk}$ can be solved
separately using the same coefficient matrix as (\ref{eq:least_squares})
and the same weighted least squares formulation,
but two different right-hand side vectors, and then $a_{j,k+1}$ and
$b_{j+1,k}$ must be averaged. The latter approach compromises the
optimality of the solution without compromising the symmetry and the
order of convergence. We choose the latter approach for simplicity.

After obtaining the coefficients $a$s and $b$s, we then obtain the
Hessian of the height function at $\vec{u}_{0}$ as \[
\vec{H}_{0}=\left[\begin{array}{cc}
a_{10} & (a_{01}+b_{10})/2\\
(a_{01}+b_{10})/2 & b_{01}\end{array}\right].\]
From the gradient and Hessian, the second-order differential quantities
are then obtained using the formulas in Section~\ref{sec:continuous-surf}.
We summarize the overall iterative fitting algorithm as follows: \\
1) Obtain initial estimation of vertex normals by averaging face normals
for the construction of local coordinate systems at vertices;\\
2) For each vertex, determine $\frac{d+1}{2}$-ring neighbor vertices
and upgrade neighborhood if necessary;\\
3) For each vertex, transform its neighbor points into its local coordinate
frame, solve for the coefficients using QR factorization with safeguard,
and convert gradients into vertex normals;\\
4) If iterative fitting is desired, for each vertex, transform normals
of its neighbor vertices to the gradient of the height function in its
local coordinate system to solve for Hessian;\\
5) For each vertex, convert Hessian into symmetric shape operator
to compute curvatures, principal directions, or curvature tensor.

In the algorithm, some steps (such as the collection of neighbor points)
may be merged into the loops in later steps, with a tradeoff
between memory and computation. Note that iterative fitting is optional,
because we found it to be beneficial only for odd-degree fittings,
as discussed in more detail in the next subsection. Without iterative
fitting, the Hessian of the height function should be obtained at
step 3. Finally, note that we can also apply the idea of iterative
fitting to convert the curvature tensor to the Hessian of the height
function and then construct a fitting of the Hessian to obtain
higher derivatives. However, since the focus of this paper is on first-
and second-order differential quantities, we do not pursue this idea
further.

\subsection{Error Analysis}

\label{sub:errorana}

To complete the discussion of our framework, we must address the fundamental
question of whether the computed differential quantities converge
as the mesh gets refined, and if so, what is the convergence rate.
While it has been previously shown that polynomial fittings produce
accurate normal and curvature estimations in noise-free contexts \cite{CP05EDQ,MW00SNG},
our framework is more general and uses different formulas 
for computing curvatures, so it requires a more
general analysis. Let $h$ denote the average edge length of the mesh.
We consider the errors in terms of $h$. Our theoretical results have
two parts, as summarized by the following two theorems.

\begin{thm}
\label{thm:coeffs}Given a set of points in $uvw$ coordinate frame
that interpolate a smooth height function $f$ or approximate $f$
with an error of $O(h^{d+1})$ along the $w$ direction. Assume the
point distribution and the weighing matrix are independent of the
mesh resolution, and the scaled matrix $\vec{A}\vec{S}$ in (\ref{eq:weighted_lsf})
has a bounded condition number. The degree-$d$ weighted least squares
fitting approximates $c_{jk}$ to $O(h^{d-j-k+1})$.
\end{thm}
\begin{proof}
Let $\vec{c}$ denote the vector composed of $c_{jk}$, the exact
partial derivatives of $f$. Let $\tilde{\vec{b}}$ denote $\vec{A}\vec{c}$.
Let $\vec{r}\equiv\vec{b}-\tilde{\vec{b}}$, which we consider as
a perturbation in the right-hand side of \[
\vec{A}\vec{c}\approx\tilde{\vec{b}}+\vec{r}.\]
Because the Taylor polynomial approximates $f$ to $O(h^{d+1})$,
and $f$ is approximated to $O(h^{d+1})$ by the given points, each
component of $\vec{f}-\vec{V}\vec{c}$ in (\ref{eq:least_squares})
is $O(h^{d+1})$. Since $\vec{r}=\vec{W}(\vec{f}-\vec{V}\vec{c})$
and the entries in $\vec{W}$ are $\Theta(1)$, each component of
$\vec{r}$ is $O(h^{d+1})$.

The error in $c_{jk}$ and in $\vec{r}$ are connected by the scaled
matrix $\vec{A}\vec{S}$. Since the point distribution are independent
of the mesh resolution, $u_{i}=\Theta(h)$ and $v_{i}=\Theta(h)$.
The entries in the column of $\vec{A}$ corresponding to $c_{jk}$
are then $\Theta(h^{j+k})$, so are the $2$-norm of the column. After
column scaling, the entries in $\vec{A}\vec{S}$ are then $\Theta(1)$,
so are the entries in its pseudo-inverse $(\vec{A}\vec{S})^{+}$.
The error of $\vec{d}$ in (\ref{eq:weighted_lsf}) is then $\delta\vec{d}=(\vec{A}\vec{S})^{+}\vec{r}$.
Because each component of $\vec{r}$ is $O(h^{d+1})$, each component
of $\delta\vec{d}$ is $O(\kappa h^{d+1})$, where $\kappa$ is the
condition number of $\vec{A}\vec{S}$ and 
is bounded by a constant by assumption. The error in $\vec{c}$
is then $\vec{S}\delta\vec{d}$, where the scaling factor associated with
$c_{jk}$ is $\Theta(1/h^{j+k})$. Therefore, the coefficient $c_{jk}$ is 
approximated to $O(h^{d-j-k+1})$.
\end{proof}

Note that our weights given in (\ref{eq:weights}) appear to 
depend on the mesh resolution. However, we can rescale it by
multiplying $h^{d/2}$ to make it $\Theta(1)$ 
without changing the solution, so  
Theorem~\ref{thm:coeffs} applies to our weighting scheme. 
Under the assumptions of Theorem~\ref{thm:coeffs},
degree-$d$ fitting approximates the gradient to $O(h^{d})$ and the
Hessian to $O(h^{d-1})$ at the origin of the local frame,
respectively.  Using the gradient and Hessian estimated from our
polynomial fitting, the estimated differential quantities have the
following property.

\begin{thm}
\label{thm:convergence}Given the position, gradient, and Hessian
of a height function that are approximated to $O(h^{d+1})$, $O(h^{d})$
and $O(h^{d-1})$, respectively. a) The angle between the computed
and exact normals is $O(h^{d})$; b) the components of the shape
operator and curvature tensor are approximated $O(h^{d-1})$ by (\ref{eq:shapeop_orth})
and (\ref{eq:curvtens_l}); c) the Gaussian and mean curvatures are
approximated to $O(h^{d-1})$ by (\ref{eq:curvatures_stable}).
\end{thm}
\begin{proof}
a) Let $\tilde{f}_{u}$ and $\tilde{f}_{v}$ denote
the estimated derivatives of $f$, which approximate the true derivatives
$f_{u}$ and $f_{v}$ to $O(h^{d})$. Let $\tilde{\ell}$ and $\ell$
denote $\Vert[-\tilde{f}_{u},-\tilde{f}_{v},1]\Vert$ and
$\Vert[-f_{u},-f_{v},1]\Vert$, respectively. Let $\tilde{\vec{n}}$
denote the computed unit normal $[-\tilde{f}_{u},-\tilde{f}_{v},1]^{T}/\tilde{\ell}$
and $\hat{\vec{n}}$ the exact unit normal $[-f_{u},-f_{v},1]^{T}/\ell$.
Therefore, 
\[
\tilde{\vec{n}}-\hat{\vec{n}}=\frac{\ell[-\tilde{f}_{u},-\tilde{f}_{v},1]^{T}-\tilde{\ell}[-f_{u},-f_{v},1]^{T}}{\ell\tilde{\ell}},\]
where the numerator is $O(h^{d})$ and the denominator is $\Theta(1)$.
Let $\theta$ denote $\arccos\tilde{\vec{n}}^{T}\hat{\vec{n}}$. Because $\Vert\tilde{\vec{n}}-\hat{\vec{n}}\Vert_{2}^{2}=2-2\tilde{\vec{n}}^{T}\hat{\vec{n}}=\theta^{2}+O(\theta^{4})$,
it then follows that $\theta$ is $O(h^{d})$.

b) In (\ref{eq:shapeop_orth}), if $f_{u}=f_{v}=0$, then $\tilde{\vec{W}}=\vec{H}$,
which is approximated to $O(h^{d-1})$. Otherwise, $\vec{H}$ is approximated
to $O(h^{d-1})$ while $\ell$, $c$, and $s$ are approximated to
$O(h^{d})$, so the error in $\tilde{\vec{W}}$ is $O(h^{d-1})$.
Similarly, in equations (\ref{eq:curvtens_l}) and (\ref{eq:curvtens_g}),
the dominating error term is the $O(h^{d-1})$ error in $\vec{H}$, while
$\vec{J}^{+}$ and $\tilde{\vec{J}}^{+}$ are approximated to $O(h^{d})$
by their explicit formulas, so $\vec{C}$ and $\vec{C}_{g}$ are approximated
to $O(h^{d-1})$.

c) In (\ref{eq:curvatures_stable}), $\ell$ and $\nabla f$ are approximated
to $O(h^{d})$ and $\vec{H}$ is approximated to $O(h^{d-1})$. Therefore,
$\kappa_{G}\ $ and $\kappa_{H}$ are both approximated to $O(h^{d-1})$.
\end{proof}
The above analysis did not consider iterative fitting. Following the
same argument, if the vertex positions and normals are both
approximated to $O(h^{d+1)})$ and the scaled coefficient matrix has
full rank, then the coefficients $a_{jk}$ and $b_{jk}$ are
approximated to order $O(h^{d-j-k+1)})$ by our iterative fitting.  The
error in the Hessian would then be $O(h^{d})$, so are the estimated
curvatures. Therefore, iterative fitting is potentially advantageous,
given accurate normals.  Note that none of our analyses requires any
symmetry of the input data points to achieve convergence. For
even-degree fittings, the leading term in the remainder of the Taylor
series is odd degree. If the input points are perfectly
symmetric, then the residual would also exhibit some degree of symmetry,
and the leading-order error term may cancel out as in the
centered-finite-difference scheme. Therefore, superconvergence may be
expected for even-degree polynomial fittings, and iterative fitting
may not be able to further improve their accuracies. Regarding 
to the principal directions, they are inherently unstable at the points where
the maximum and minimum curvatures have similar magnitude. However, if
the magnitudes of the principal curvatures are well separated, then
the principal directions would also have similar convergence rates as
the curvatures.

\section{Experimental Results}
\label{sec:Experimental-Study}

In this section, we present some experimental results of our
framework.  We focus on the demonstration of accuracy and stability
as well as the advantages of iterative fitting, the
weighting scheme, and the safeguarded numerical solver. We do not
attempt a thorough comparison with other methods; readers are referred
to \cite{GG06ECT} for such a comparison of earlier methods. We
primarily compare our method against the baseline fitting methods in
\cite{CP05EDQ,MW00SNG}, and assess them for both closed and open
surfaces.

\subsection{Experiments with Closed Surfaces}

We first consider two simple closed surfaces: a sphere with unit radius,
and a torus with inner radius $0.7$ and outer radius $1.3$. We generated
the meshes using GAMBIT, a commercial software from Fluent Inc. Our
focuses here are the convergence rates with and without iterative
fitting as well as the effects of the weighting scheme. For convergence
test, we generated four meshes for each surface independently of each
other by setting the desired edge lengths to $0.1,$ $0.05$, $0.025$,
and $0.0125$, respectively. Figure~\ref{fig:closed-surface} shows
two meshes that are coarser but have similar unstructured connectivities
as our test meshes.

\begin{figure}[t!]
\begin{centering}
\includegraphics[width=0.49\columnwidth]{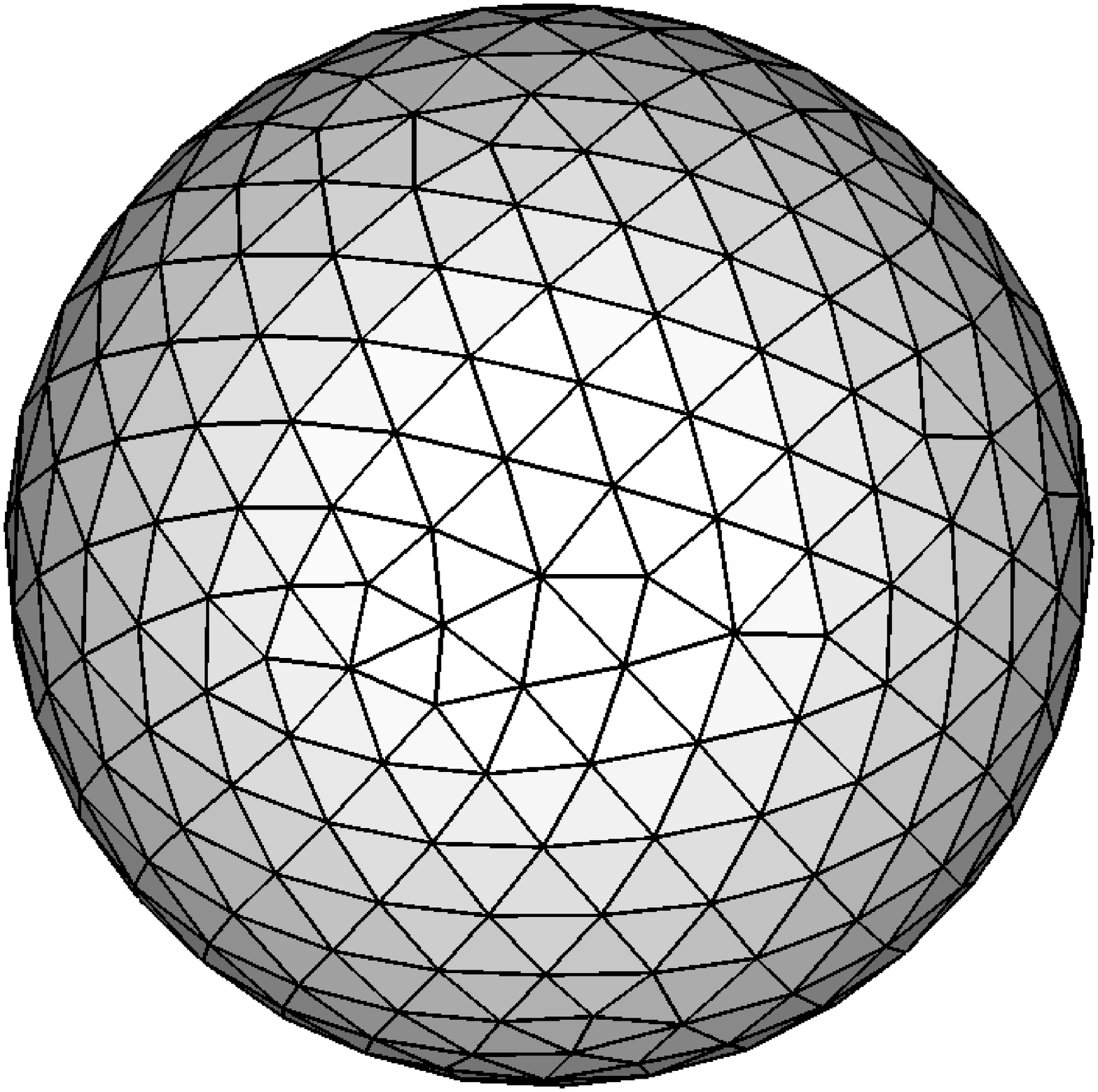}\hfill\includegraphics[width=0.49\columnwidth]{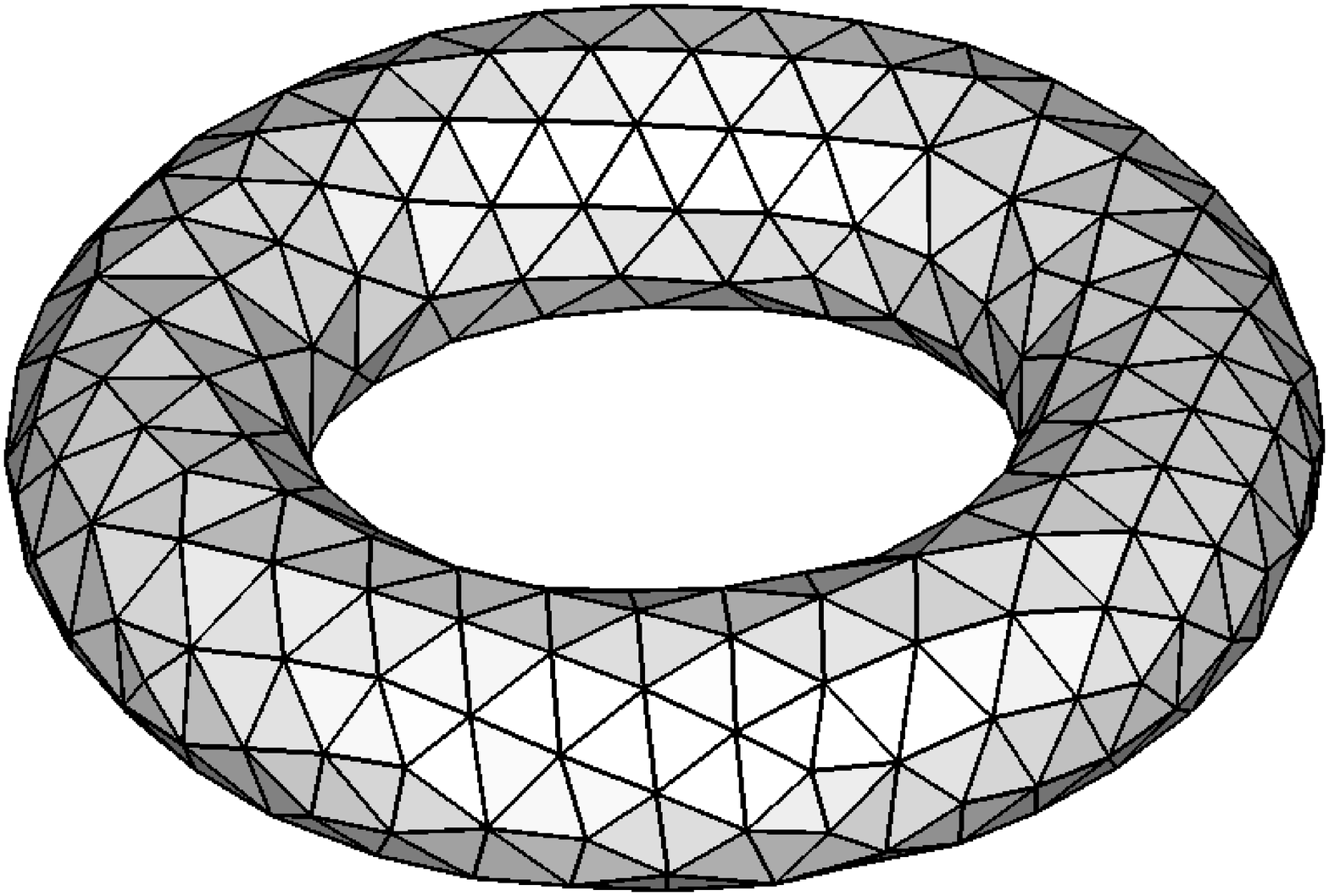}
\par\end{centering}

\caption{\label{fig:closed-surface}Sample unstructured meshes of sphere and torus.}

\end{figure}

We first assess the computations of normals using fittings of degrees
between one and six. Figure~\ref{fig:nrmerrs} shows the errors in
the computed normals versus the ``mesh refinement level.''
We label the plots by the degrees
of fittings. Let $v$ denote the total number of vertices, and let
$\hat{\vec{n}}_{i}$ and $\tilde{\vec{n}}_{i}$ denote the exact and 
computed unit vertex normals at the $i$th vertex. We measure the relative
$L_{2}$ errors in normals as\[
\sqrt{\frac{1}{v}\sum_{1}^{v}\Vert\tilde{\vec{n}}_{i}-\hat{\vec{n}}_{i}\Vert_{2}^{2}}.\]
We compute the convergence rates as \[
\mbox{convergence rate}=\frac{1}{3}
\log_{2}\left(\frac{\mbox{error of level 1}}{\mbox{error of level 4}}\right),\]
and show them at the right ends of the curves. In our tests, the
convergence rates for normals were equal to or higher than the degrees
of fittings. For spheres, the convergence rates of even-degree fittings
were about one order higher than predicted, likely due to nearly perfect
symmetry and error cancellation.

\begin{figure}[tbh]
\begin{centering}
\includegraphics[width=0.49\columnwidth]{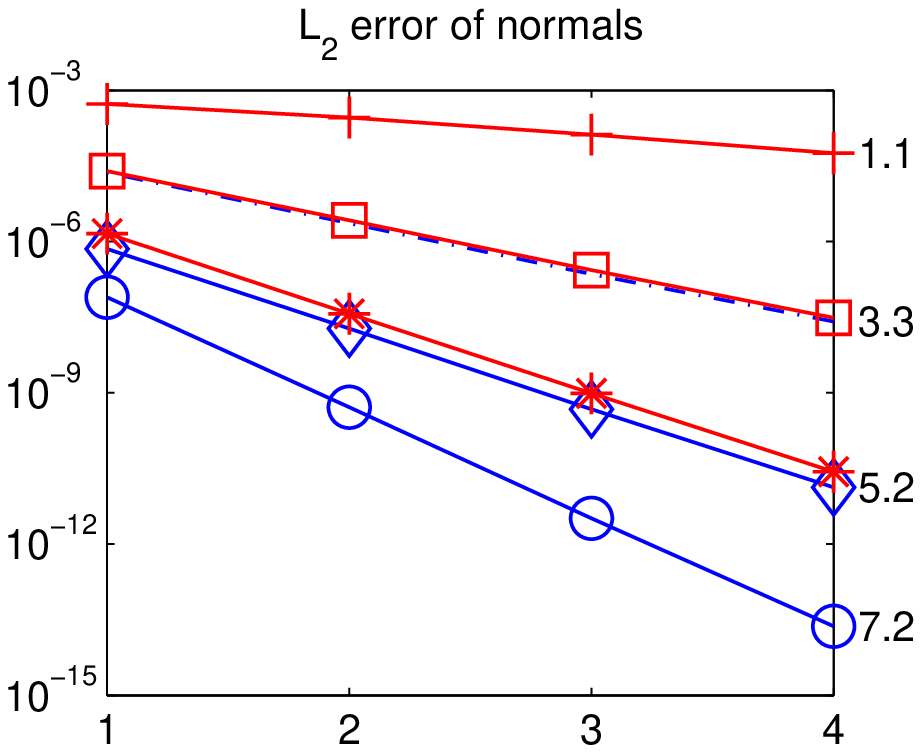}\hfill\includegraphics[width=0.49\columnwidth]{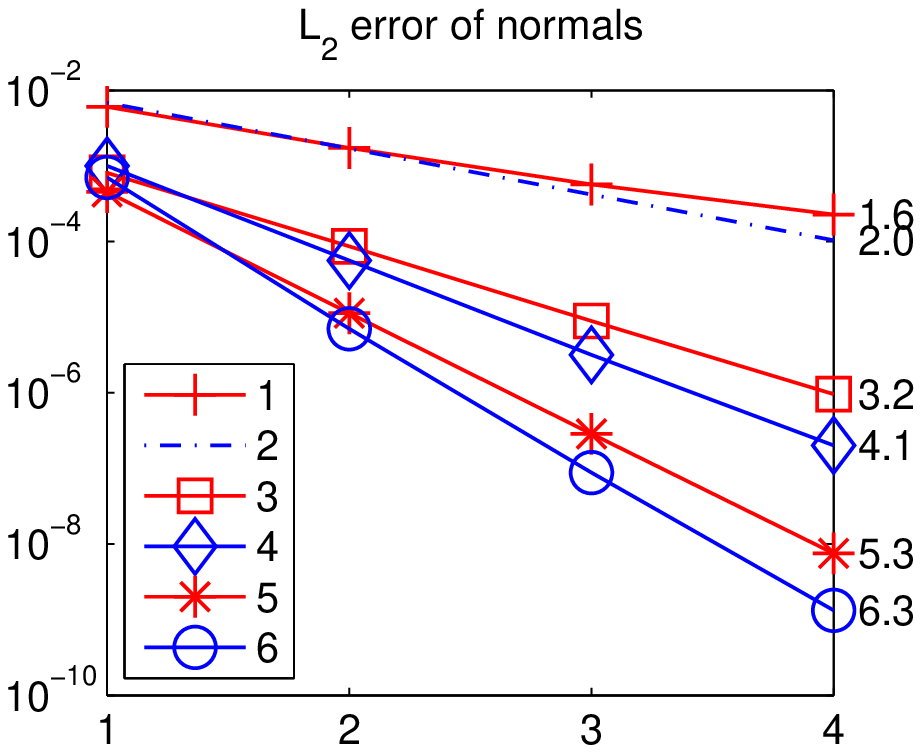}
\par\end{centering}

\caption{\label{fig:nrmerrs}$L_2$ errors in computed normals for 
sphere (left) and torus (right). }

\end{figure}

Second, we consider the computations of curvatures. It would be excessive
to show all the combinations, so we only
show some representative results. Figure~\ref{fig:sphere} shows
the errors in the minimum and maximum curvatures for the sphere. Figure~\ref{fig:torus}
shows the errors in the mean and Gaussian curvatures for the torus. In
the labels, `$+$' indicates the use of iterative fitting.
Let $k_i$ and $\tilde{k}_i$ denote the exact and computed 
quantities at the $i$th vertex, we measure the relative errors in $L_{2}$ norm as
\begin{equation}
\frac{\Vert\tilde{\kappa}-\kappa\Vert_{2}}{\Vert\kappa\Vert_{2}}\equiv
\left.
\sqrt{\sum_{i=1}^{v}\left(\tilde{\kappa}_{i}-\kappa_{i}\right)^{2}}\right/\sqrt{\sum_{i=1}^{v}\kappa_{i}^{2}}.\label{eq:l2err}\end{equation}
Let $d$ denote the degree of a fitting. In these tests, the convergence
rates were $d-1$ or higher as predicted by theory. In addition,
even-degree fittings converged up to one order faster due to error
cancellation. The converge rates for odd-degree polynomials were
about $d-1$ but were also boosted to approximately $d$ when iterative
fitting is used. Therefore, iterative fitting is effective in 
improving odd-degree fittings. In our experiments,
iterative fitting did not improve even-degree fittings.

\begin{figure}[tbh]
\begin{centering}
\includegraphics[width=0.49\columnwidth]{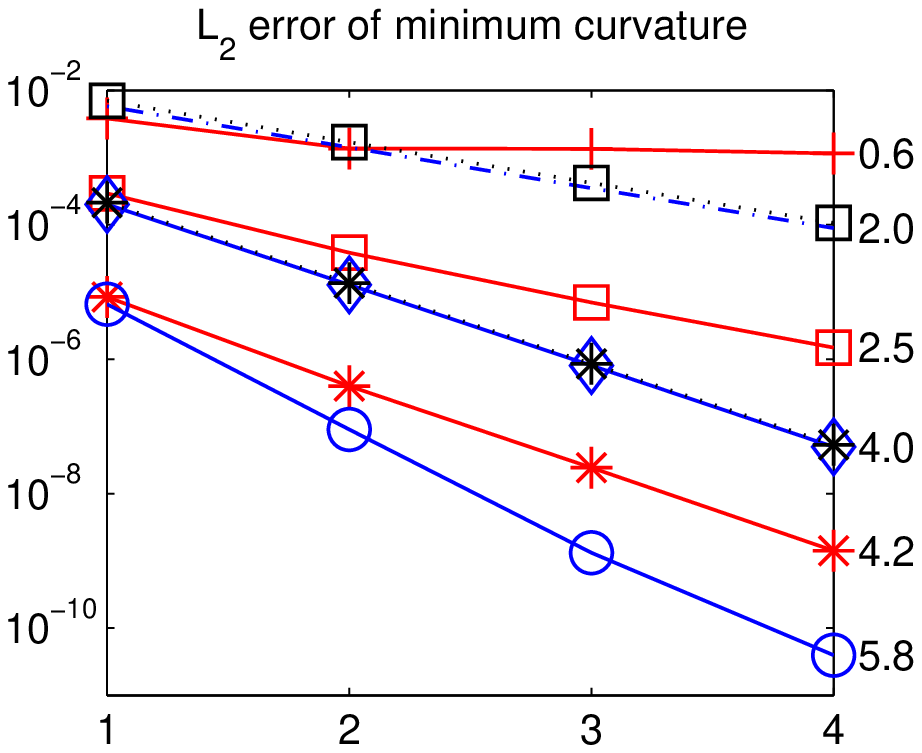}\hfill\includegraphics[width=0.49\columnwidth]{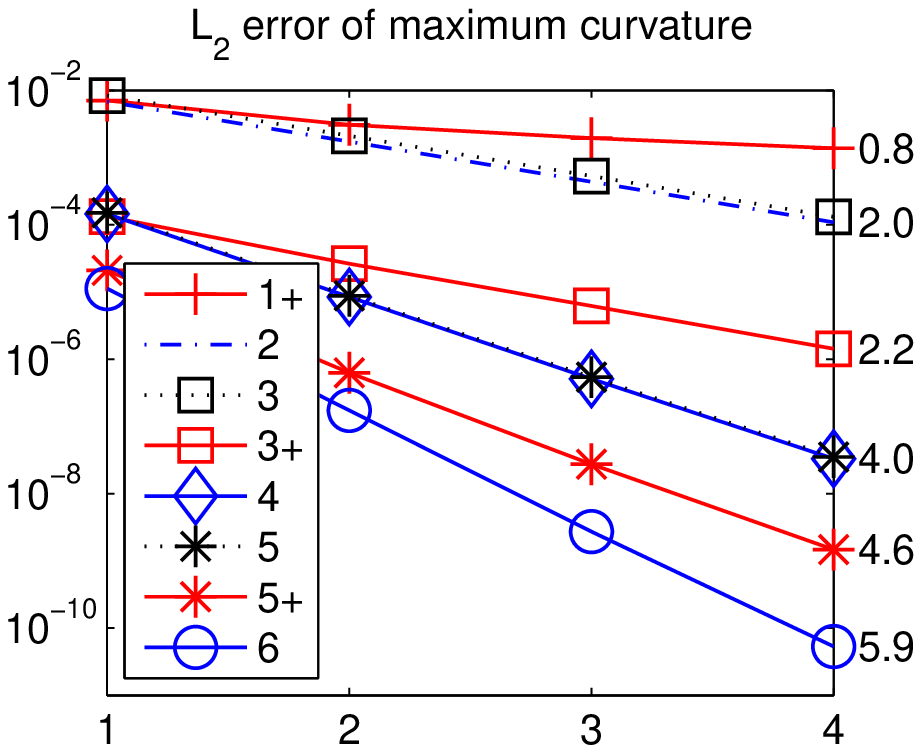}
\par\end{centering}

\caption{\label{fig:sphere}$L_2$ errors in computed minimum and maximum 
curvatures for sphere.}

\end{figure}

\begin{figure}[tbh]
\begin{centering}
\includegraphics[width=0.49\columnwidth]{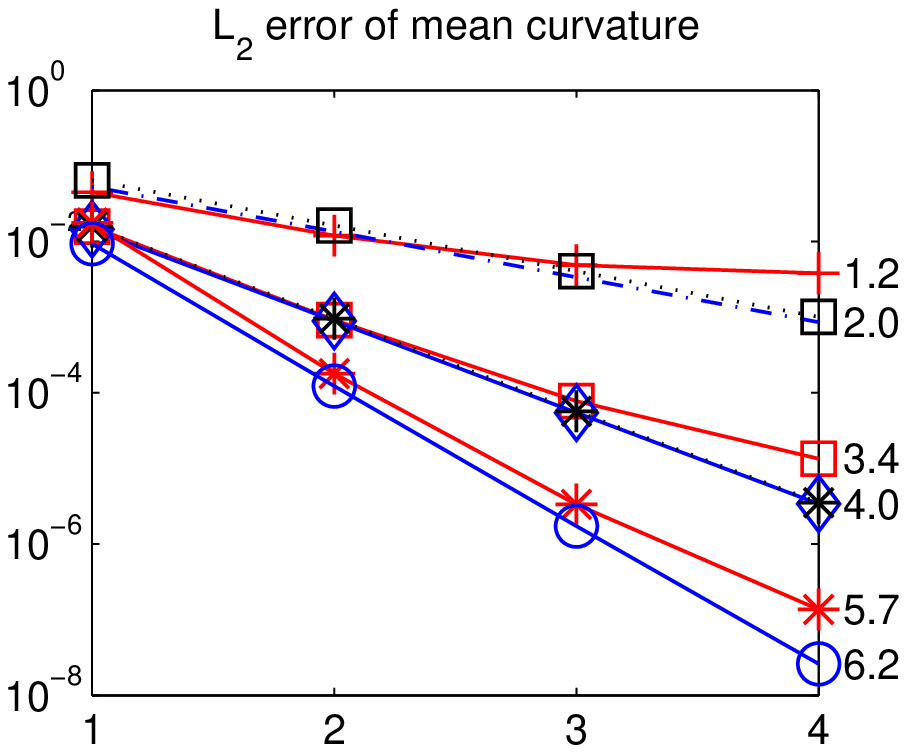}\hfill\includegraphics[width=0.49\columnwidth]{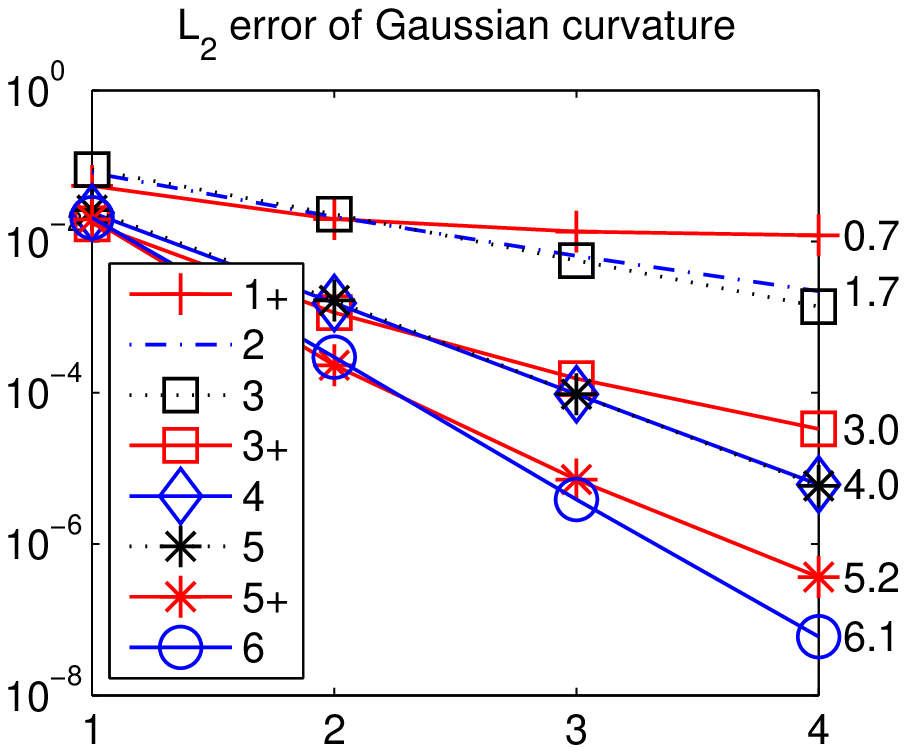}\caption{\label{fig:torus}$L_2$ errors in computed mean 
and Gaussian curvatures for torus.}

\par\end{centering}
\end{figure}

The preceding computations used the weighting scheme described in
Section~\ref{sub:local_poly}, which tries to balance conditioning
and error cancellation. This weighting scheme improved the results
in virtually all of our tests. Figure~\ref{fig:noweighting-torus} shows a representative
comparison with and without weighting (as labeled by ``$\cdot$nw''
and ``$\cdot$w'', respectively) for the maximum curvatures of
the sphere and torus.

\begin{figure}[tbh]
\begin{centering}
\includegraphics[width=0.49\columnwidth]{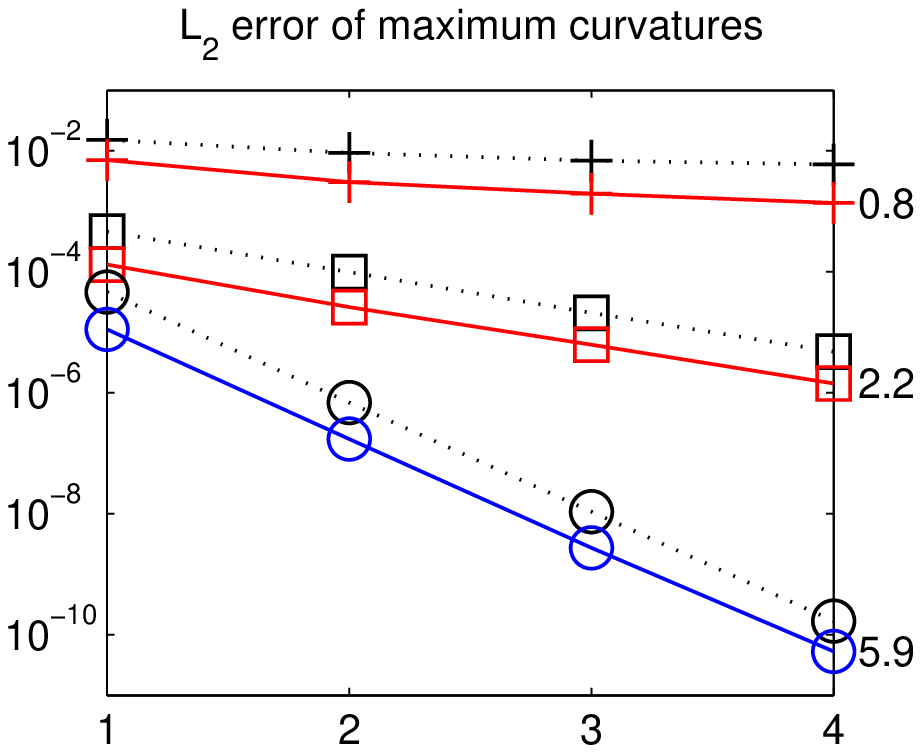}\hfill\includegraphics[width=0.49\columnwidth]{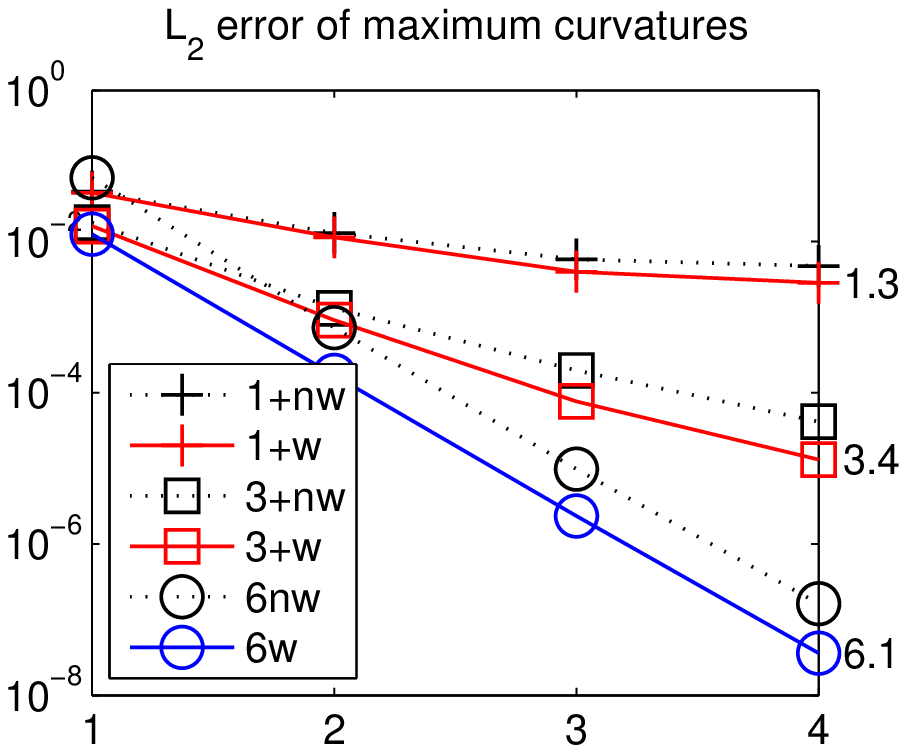}
\par\end{centering}

\caption{\label{fig:noweighting-torus}Comparisons of curvature computations
with and without weighting for sphere (left) and torus (right).}

\end{figure}

\subsection{Experiments with Open Surfaces}

We now consider open surfaces, i.e. surfaces with boundary. We focus
on the study of stabilities and the effects of boundary and irregular
connectivities. We use two surfaces defined by the following functions
adopted from \cite{Xu04CDL}: \begin{align}
z = F_{1}(x,y) = & \frac{1.25+\cos(5.4y)}{6+6(3x-1)^{2}},\label{eq:f1}\\
z = F_{2}(x,y) = & \exp\left(-\frac{81}{16}\left((x-0.5)^{2}+(y-0.5)^{2}\right)\right),\label{eq:f2}\end{align}
where $(x,y)\in[0,1]\times[0,1]$. Figure~\ref{fig:Test-cases}(a-b)
shows these surfaces, color-coded by the mean curvatures. We use two
types of meshes, including irregular and semi-regular meshes (see
Figure~\ref{fig:Test-cases}(c-d)). For convergence study, we refine
the irregular meshes using the standard one-to-four subdivision \cite[p. 283]{Gal00CSG}
and refine the semi-regular meshes by replicating the pattern. We computed
the ``exact'' differential quantities using the formulas in Section~\ref{sec:continuous-surf}
in the global coordinate system, but performed all other computations
in local coordinate systems. For rigorousness of the tests, we consider both $L_{2}$
and $L_{\infty}$ errors. In addition, border vertices are included in
all the error measures, posting additional challenges to the tests.
Note that the results for vertices far away from the boundary would be
qualitatively similar to those of closed surfaces. We primarily consider 
fittings of up to degree four, since higher convergence rates may 
require larger neighborhoods for border vertices (more than $3.5$-ring
neighbors). To limit the length of presentation, we report only some
representative results to cover the aforementioned different aspects.

\begin{figure}[tbh]
\begin{centering}
\subfigure[$F_1$.]{\includegraphics[width=0.49\columnwidth]{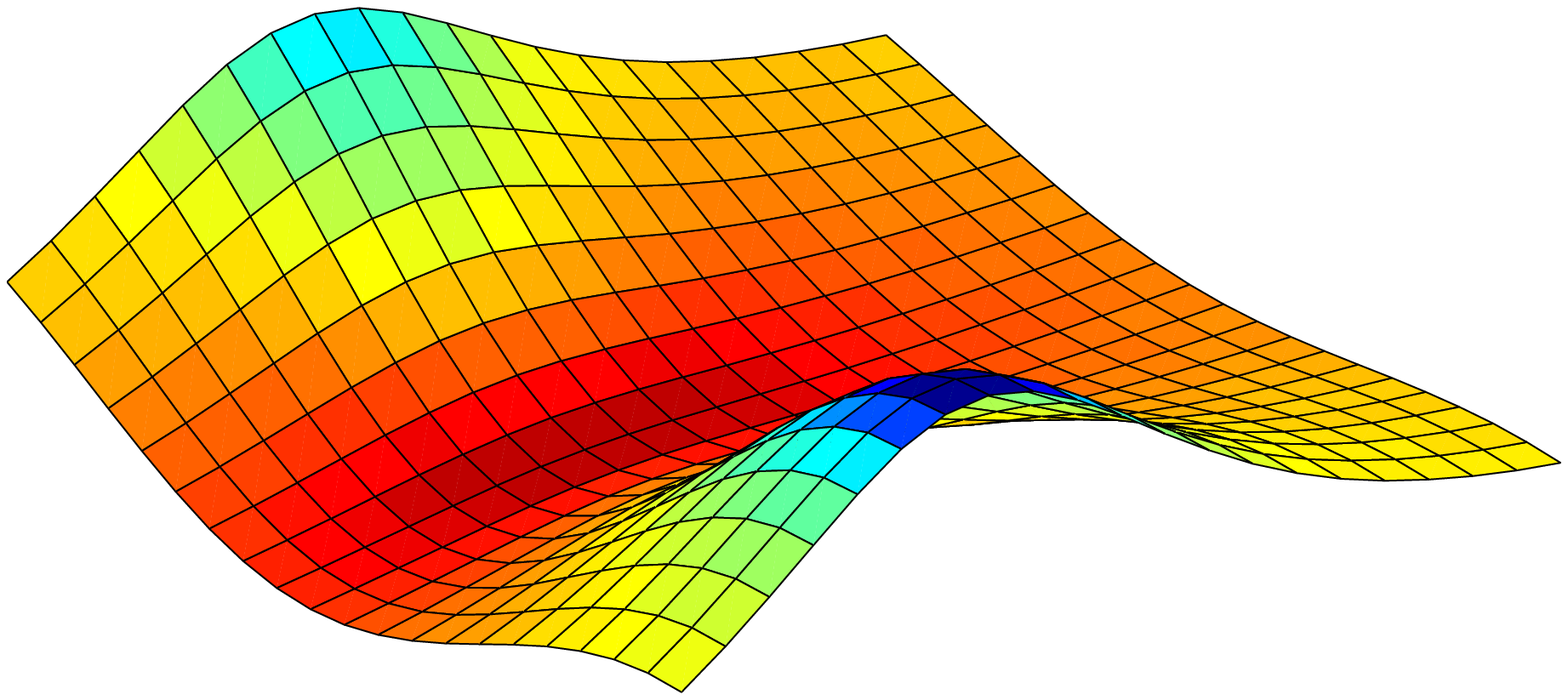}}
\hfill\subfigure[$F_2$.]{\includegraphics[width=0.49\columnwidth]{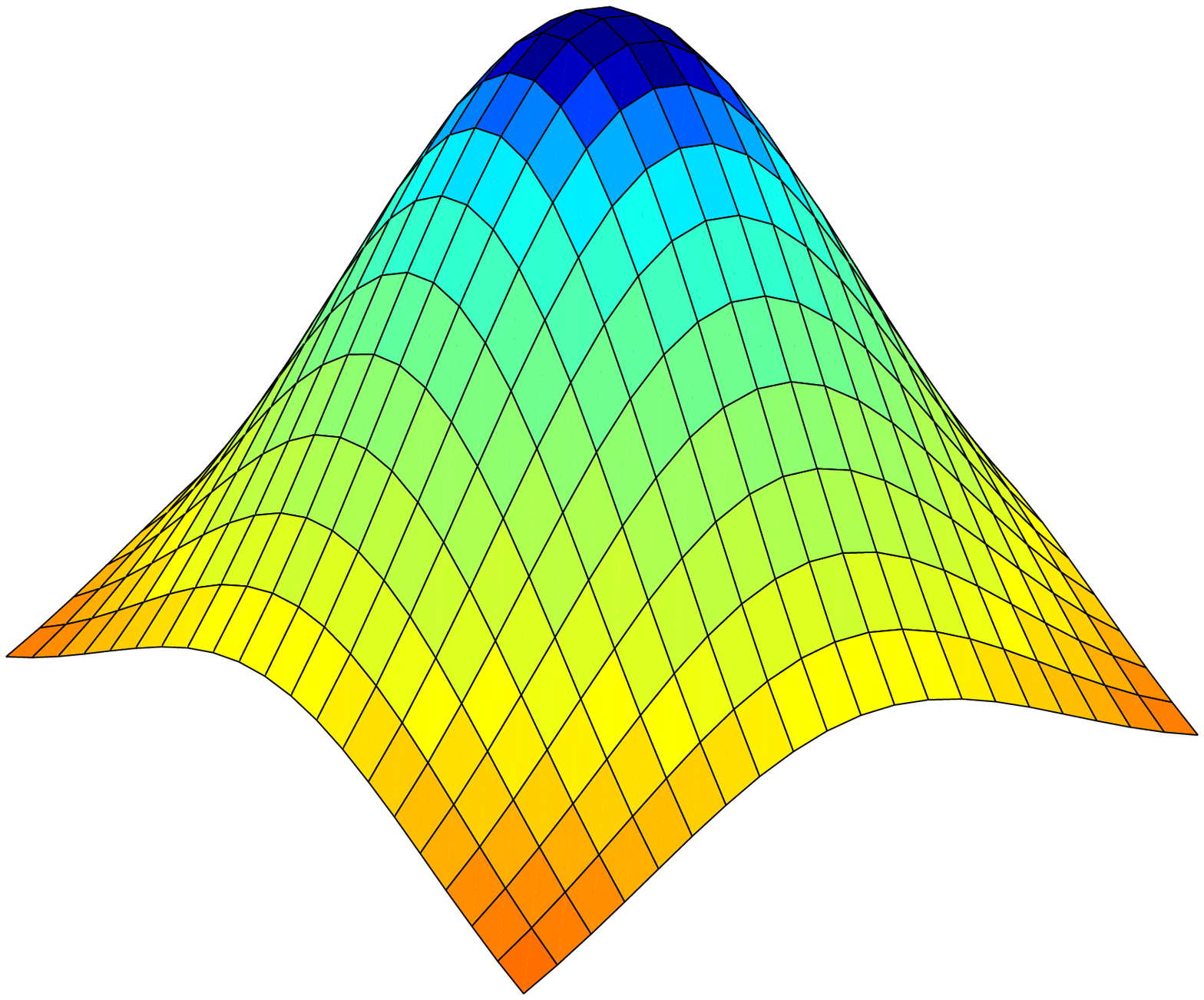}}\\
\subfigure[Irregular mesh.]{\includegraphics[width=0.49\columnwidth]{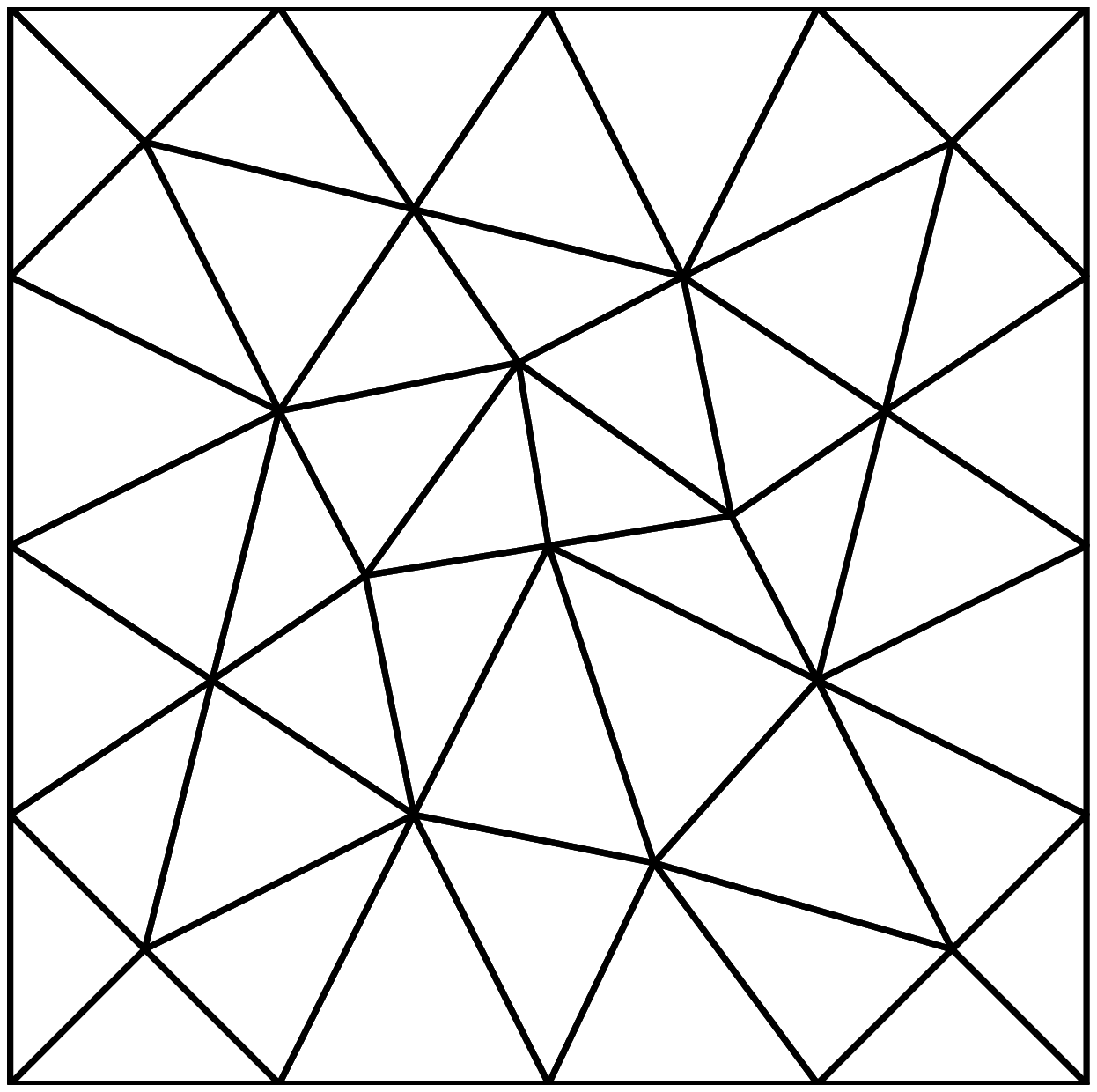}}\hfill
\subfigure[Semi-regular mesh.]{\includegraphics[width=0.49\columnwidth]{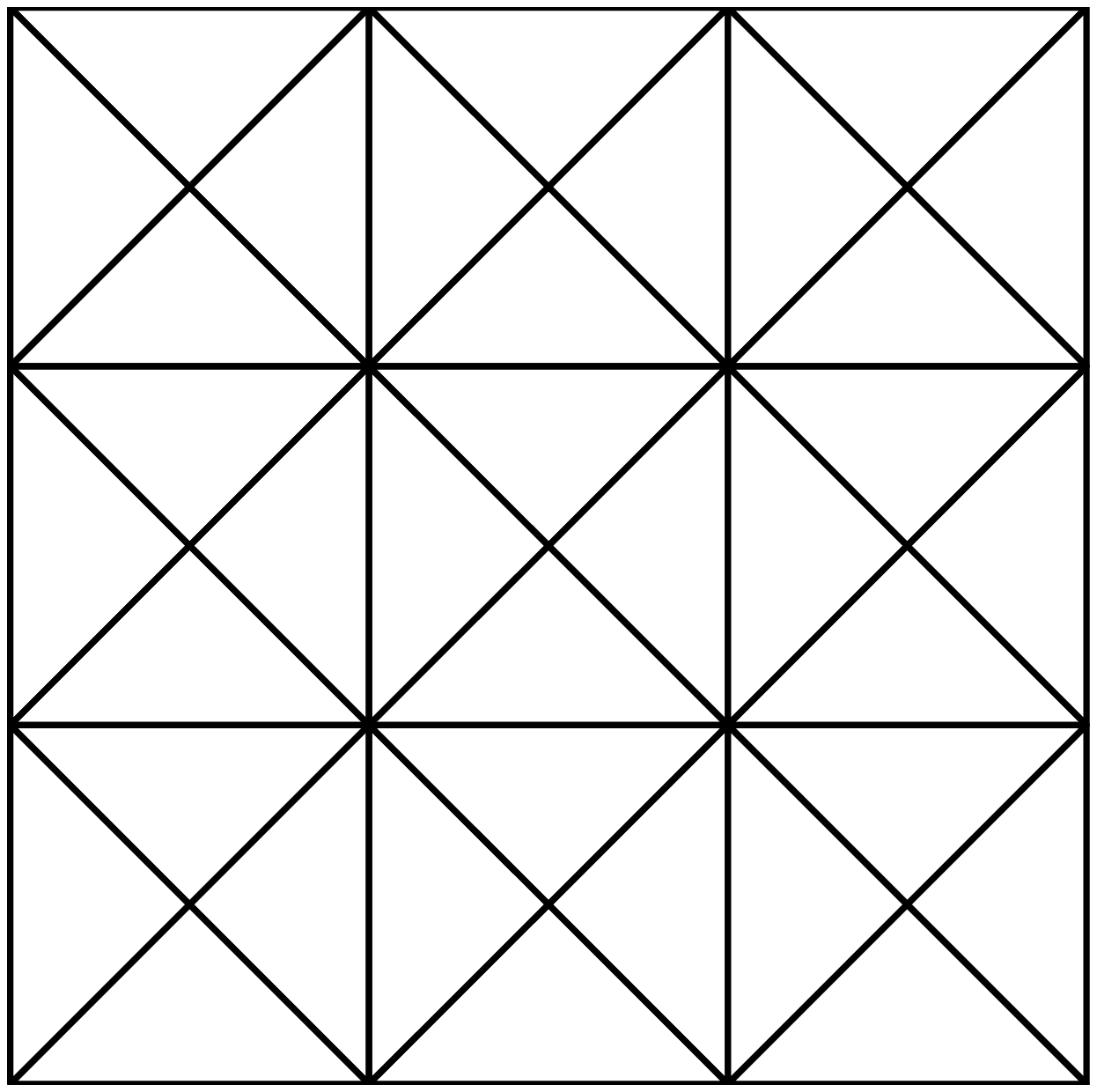}} 
\par\end{centering}

\caption{\label{fig:Test-cases}Test surfaces (a-b) and meshes (c-d). Surfaces
are color-coded by mean curvatures.}

\end{figure}

We first assess the errors in the computed normals for open surfaces.
Figure~\ref{fig:f1_nrm} shows the results for $F_{1}$ over irregular
meshes. We label all the plots and convergence rates in the same way
as for closed surfaces. Let $d$ denote the degree of a fitting.  All
these fittings delivered convergence rate of $d$ or higher in $L_{2}$
errors.  In $L_{\infty}$ errors, computed as
$\max_{i}\Vert\tilde{\vec{n}}_{i}-\hat{\vec{n}}_{i}\Vert_{2}$, the
convergence rates were $d-0.5$ or higher, close to theoretical predictions.

\begin{figure}[tbh]
\begin{centering}
\includegraphics[width=0.49\columnwidth]{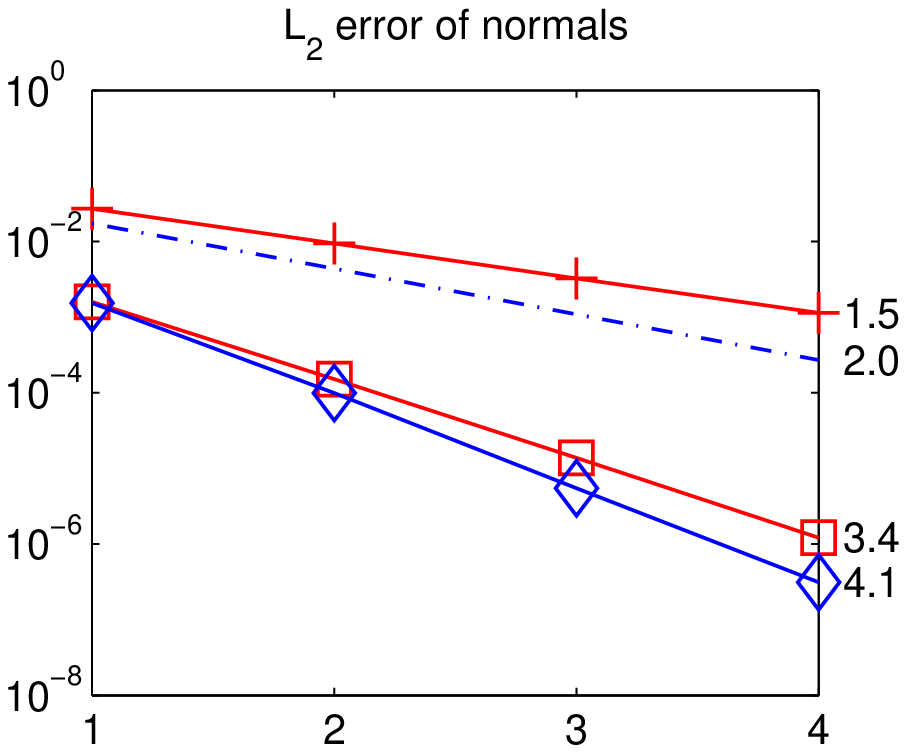}\includegraphics[width=0.49\columnwidth]{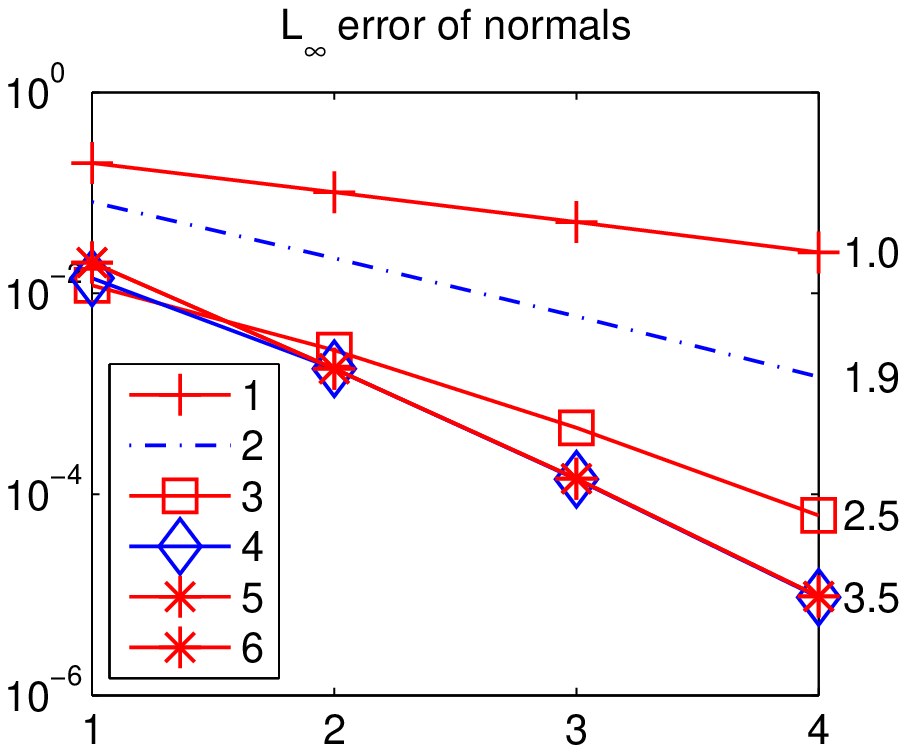}
\par\end{centering}

\caption{\label{fig:f1_nrm}$L_2$ (left) and $L_\infty$ (right) errors in computed normals for $F_{1}$ over irregular meshes.}

\end{figure}

Second, we assess the errors in the curvatures. Figure~\ref{fig:f1_ir}
shows the errors in mean curvatures for $F_{1}$ over irregular meshes,
and Figure~\ref{fig:f2_48} shows the errors in Gaussian curvatures
for $F_{2}$ over semi-regular meshes. Let $\kappa$ and $\tilde{\kappa}$
denote the exact and computed quantities. We computed the $L_{2}$
error using (\ref{eq:l2err}) and computed the $L_{\infty}$ errors
as \begin{equation}
\max_{i}\left|\tilde{\kappa}_{i}-\kappa_{i}\right|/\max\{\left|\kappa_{i}\right|,\epsilon\},\label{eq:linferr}\end{equation}
where $\epsilon=0.01\max_{i}\left|\kappa_{i}\right|$ was introduced
to avoid division by too small numbers. The convergence rates for
curvatures were approximately equal to $d-1$ or higher for 
even-degree fittings and odd-degree iterative fittings.

\begin{figure}[tbh]
\begin{centering}
\includegraphics[width=0.49\columnwidth]{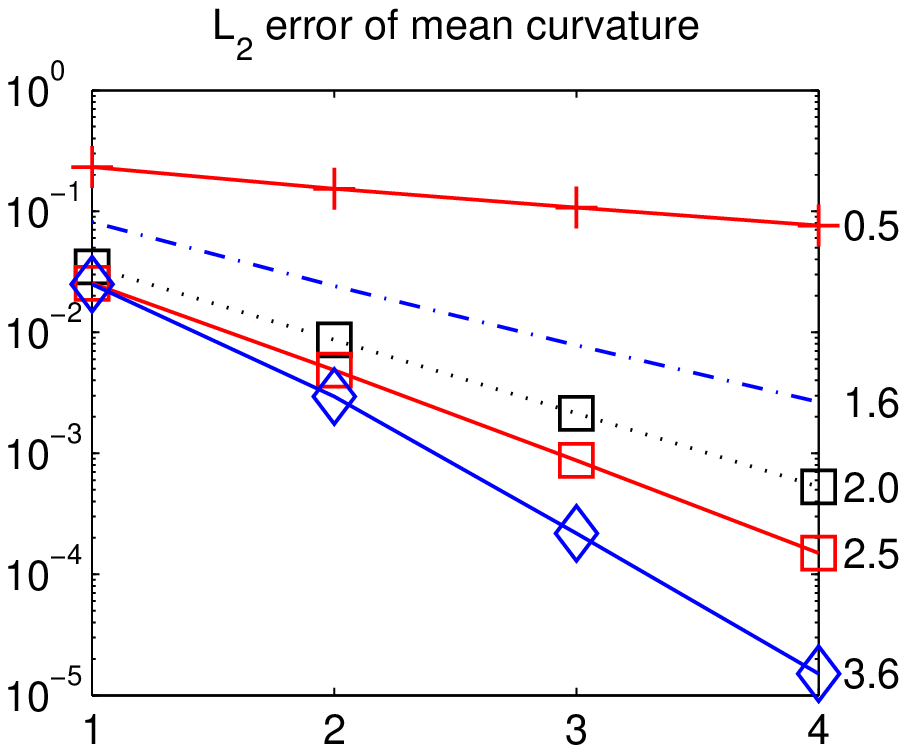}\includegraphics[width=0.49\columnwidth]{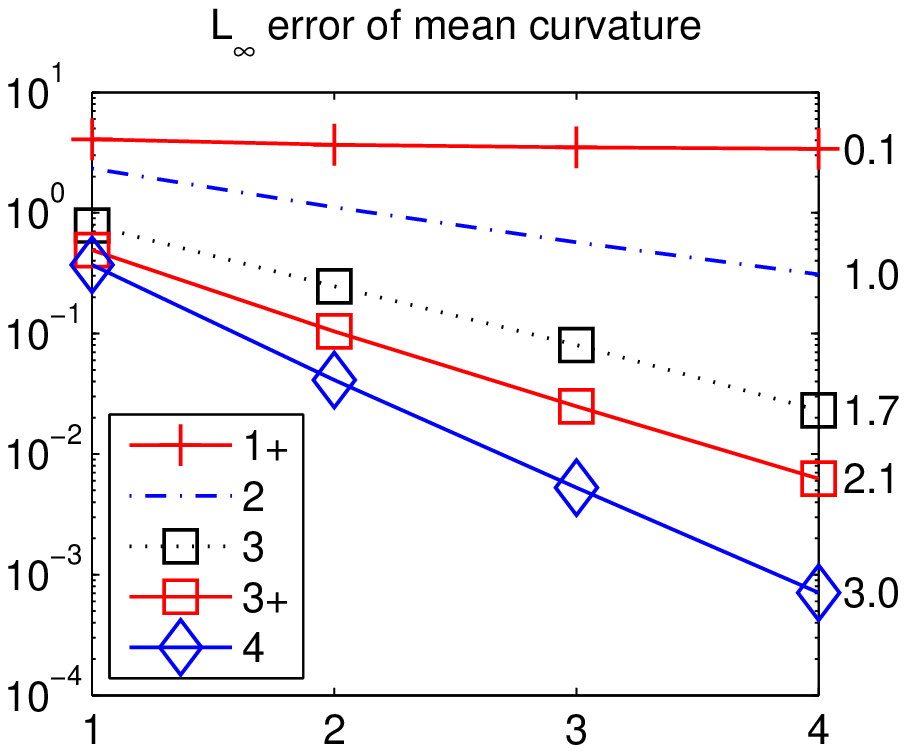}
\par\end{centering}

\caption{\label{fig:f1_ir}$L_2$ (left) and $L_\infty$ (right) in computed mean curvatures for $F_{1}$
over irregular meshes.}

\end{figure}

\begin{figure}[tbh]
\begin{centering}
\includegraphics[width=0.49\columnwidth]{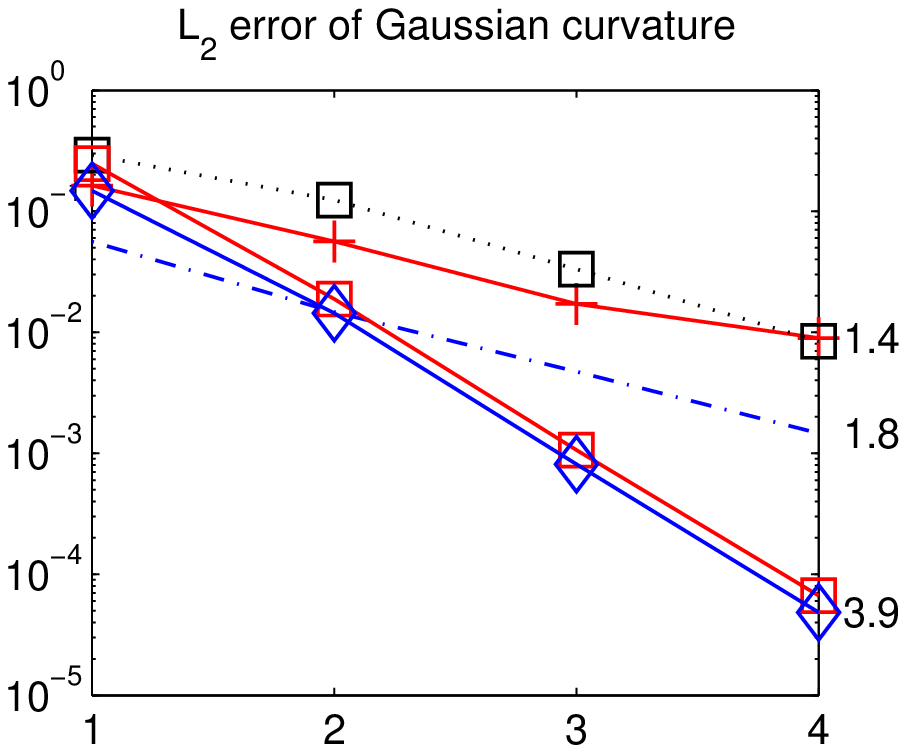}\includegraphics[width=0.49\columnwidth]{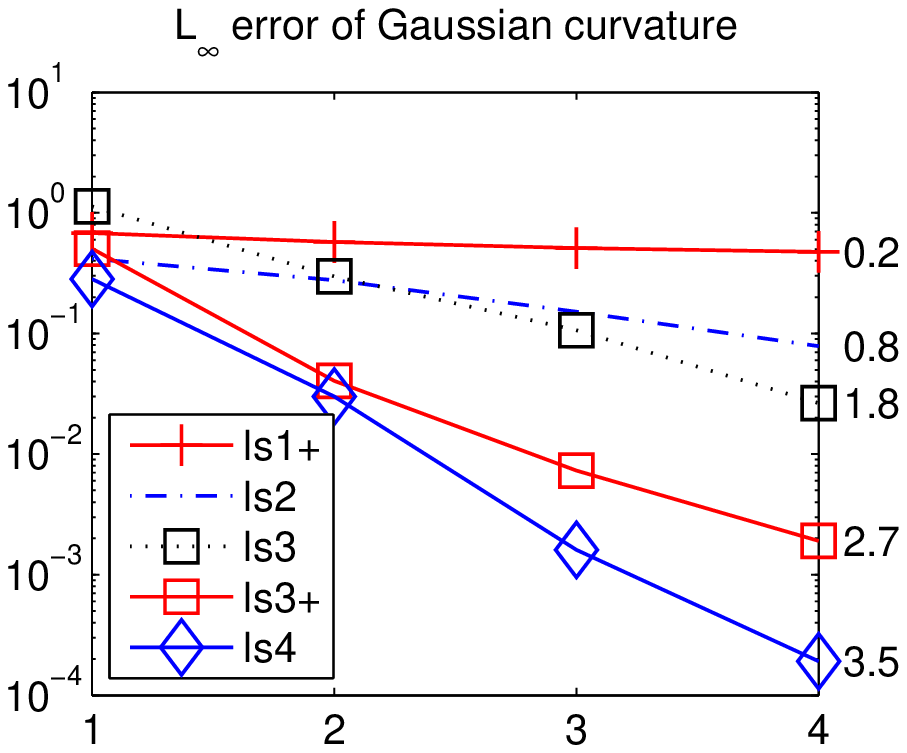}
\par\end{centering}

\caption{\label{fig:f2_48}$L_2$ (left) and $L_\infty$ (right) in computed Gaussian curvatures for $F_{2}$
over semi-regular meshes.}

\end{figure}

As noted earlier, the principal directions are inherently unstable
when the principal curvatures are roughly equal to each other (such as
at umbilic points). In Figure~\ref{fig:f1_pdr}, we show the $L_2$ and
$L_\infty$ errors of principal directions for $F_{1}$ over
semi-regular meshes, where the errors are measured similarly as for
normals. This surface is free of umbilic points, and the principal
directions converged at comparable rates as curvatures for iterative
cubic fitting and quartic fitting.

\begin{figure}[tbh]
\begin{centering}
\includegraphics[width=0.49\columnwidth]{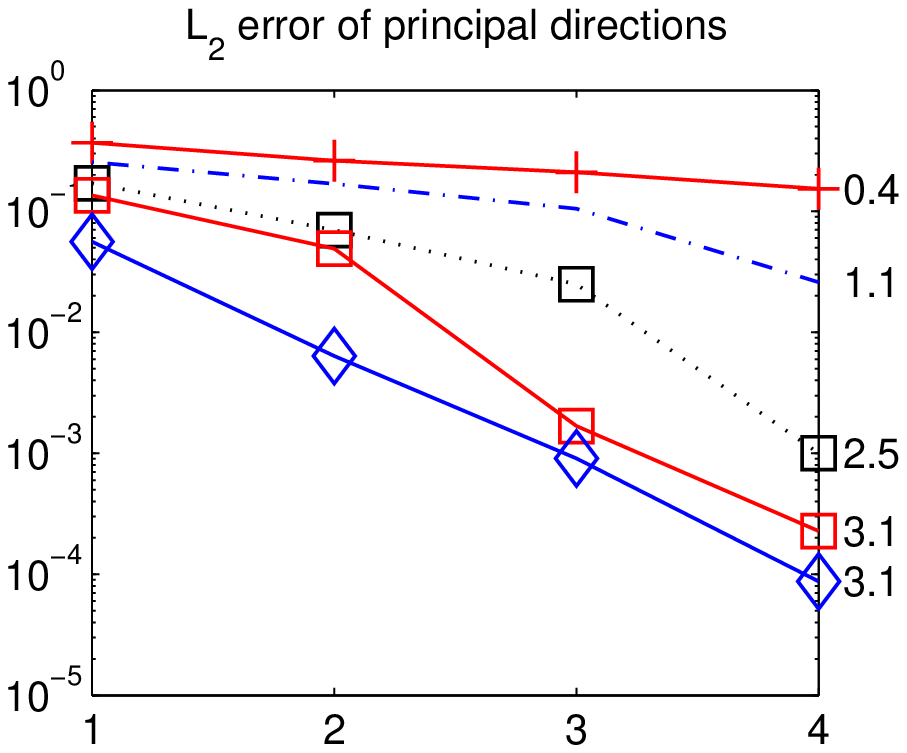}\includegraphics[width=0.49\columnwidth]{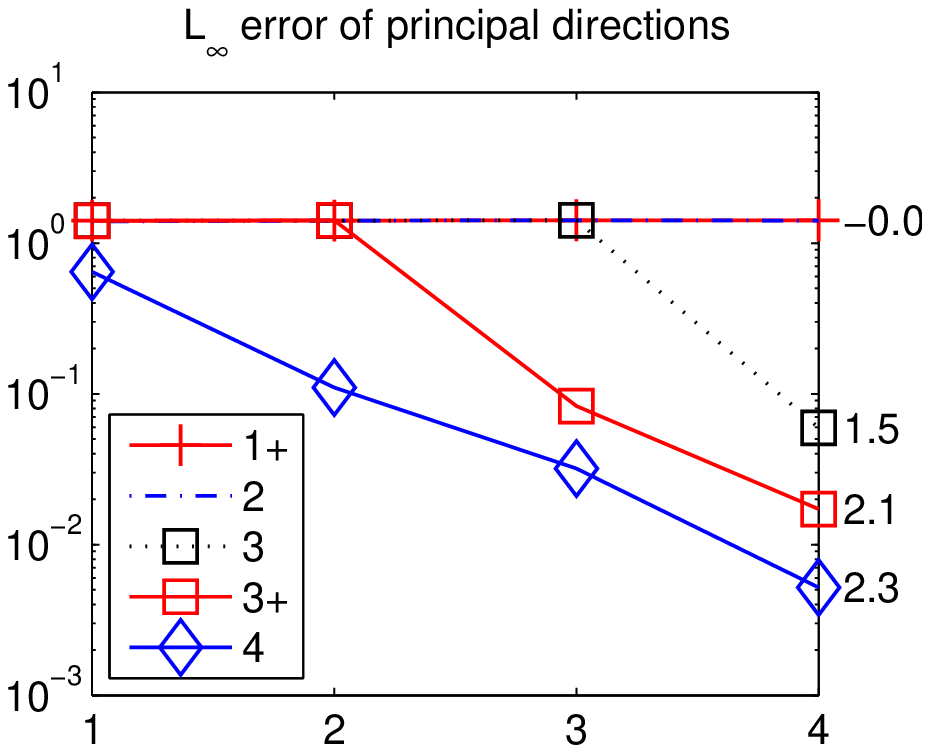}
\par\end{centering}

\caption{\label{fig:f1_pdr}$L_2$ (left) and $L_\infty$ (right) in computed principal directions for $F_{1}$
over semi-regular meshes.}

\end{figure}

Finally, to demonstrate the importance and effectiveness of our
conditioning procedure, Figure~\ref{fig:cond_comp}
shows comparison of computed normals and mean curvatures with and
without conditioning (as labeled by ``$\cdot$nc'' and ``$\cdot$c'',
respectively) for fittings of degrees three, five, and six.  Here, the
conditioning refers to requiring number of points to be $1.5$ or more times
of the number of unknowns as well as the checking of condition
numbers. Without conditioning, the results exhibited large errors for
normals and catastrophic failures for curvatures, due to numerical
instabilities.  With conditioning, our framework is stable for all the
tests, although it did not achieve the optimal convergence rate for
the sixth-degree fittings due to too small neighborhoods for the vertices
near boundary.

\begin{figure}[tbh]
\begin{centering}
\includegraphics[width=0.49\columnwidth]{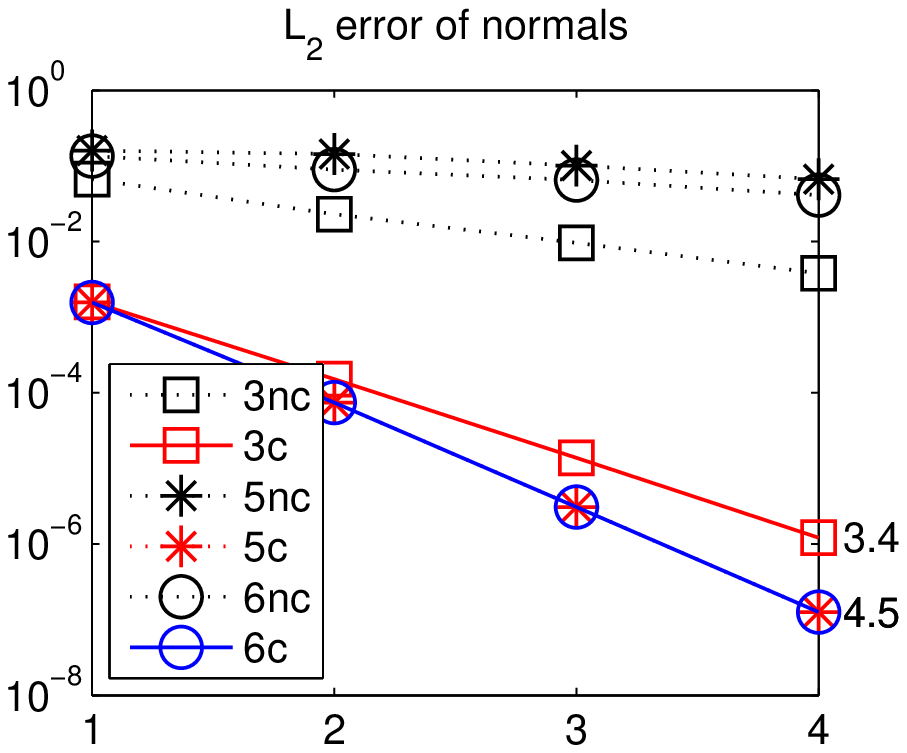}\includegraphics[width=0.49\columnwidth]{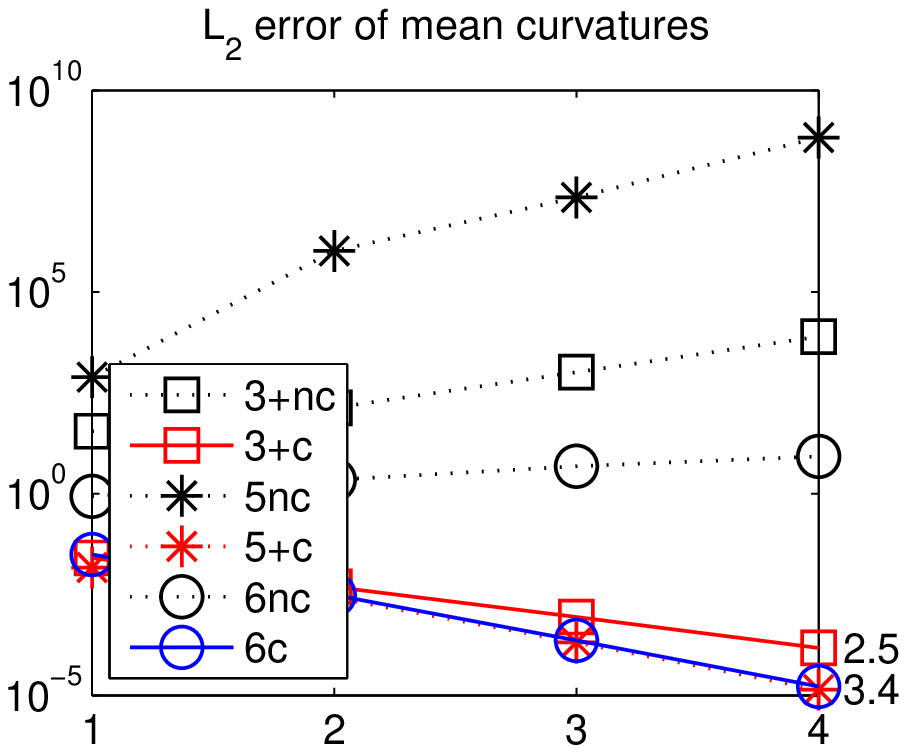}
\par\end{centering}

\caption{\label{fig:cond_comp}Comparisons of errors in computed 
normals (left) and mean curvatures (right) with and without conditioning 
for $F_{1}$ over irregular meshes.}

\end{figure}

\section{Discussions}

\label{sec:Conclusion}

In this paper, we presented a computational framework for computing the
first- and second-order differential quantities of a surface
mesh. This framework is based on the computation of the first- and
second-order derivatives of a height function, which are then
transformed into differential quantities of a surface in a simple and
consistent manner. We proposed an iterative fitting method to compute
the derivatives of the height function starting from the points with
or without the surface normals, solved by weighted least squares
approximations. We improve the numerical stability by a systematic
point-selection strategy and QR factorization with safeguard. By
achieving both accuracy and stability, our method delivered converging
estimations of the derivatives of the height function and in turn the
differential quantities of the surfaces.

%In this paper, we have used the notion of backward stability as a
%stronger version of the consistency requirement. Backward stability
%has proven to be useful for fundamental numerical algorithms, 
%because the combination of backward-stable algorithms also tend 
%to be backward stable. We believe this notion is relevant in the
%computation of differential quantities because they are often used as
%building blocks of more complex algorithms, such as geometric flows.
%The backward stability we discussed in this paper are local at each
%point. We plan to further investigate the backward stability in a global 
%sense (i.e., to ensure that the estimated quantities are exact answers 
%to a unique nearby smooth surface).

The main focus of this paper has been on the
consistent and converging computations of differential quantities. We
did not address the robustness issues for input surface meshes with
large noise and singularities (such as sharp ridges and corners).
% neither did we report a comprehensive comparison with some state-of-the-art methods. 
We have conducted some preliminary comparisons with other
methods, which we will report elsewhere. One of the major motivating
application of this work is provably accurate and stable solutions of
geometric flows for geometric modeling and physics-based
simulations. We are currently investigating the stability of our
proposed methods for such problems. Another future direction is to
generalize our method to compute higher-order differential quantities.

\section*{Acknowledgements}

This work was supported by the National Science Foundation under award
number DMS-0713736 and in part by a subcontract from the Center for
Simulation of Advanced Rockets of the University of Illinois at
Urbana-Champaign funded by the U.S. Department of Energy through the
University of California under subcontract B523819. We thank anonymous
referees for their helpful comments.

\bibliographystyle{acmsiggraph}
%\nocite{*}
\bibliography{diffops}

\end{document}